\documentclass{amsart}

\usepackage[T1]{fontenc}
\usepackage{enumerate, amsmath, amsfonts, amssymb, amsthm, dsfont, mathrsfs, wasysym, graphics, graphicx, xcolor, url, hyperref, hypcap, xargs, xstring, multicol, pdflscape, multirow, hvfloat, array, ae, aecompl, pifont, mathtools, a4wide, float, blkarray, overpic, nicefrac, stmaryrd, anyfontsize, yfonts, fontawesome}
\usepackage{pdflscape}
\usepackage{tikz}
\usepackage{tikz-3dplot}
\usepackage{bbm, enumerate, paralist}
\usepackage[noabbrev,capitalise]{cleveref}
\usepackage[normalem]{ulem}
\usepackage{marginnote}
\usepackage{enumitem}
\usepackage{svg}
\usepackage{cleveref}
\usepackage{mathdots}
\usepackage{tabularx}
\usepackage{makecell}
\usepackage{booktabs}
\usepackage{rotating}

\usepackage{listings, color}

\makeatletter
\renewcommand\paragraph{\@startsection
  {paragraph}{4}%
  \z@
  {-6\p@\@plus-2\p@\@minus-.5\p@}%
  {-1em}%
  {\normalfont\bfseries}}
\makeatother

\hypersetup{colorlinks=true, citecolor=darkblue, linkcolor=darkblue}
\usepackage[all]{xy}
\usepackage{tikz}
\usepackage{tikz-cd}
\usetikzlibrary{trees, decorations, decorations.pathmorphing, decorations.markings, decorations.shapes, shapes, arrows, matrix, calc, fit, intersections, patterns, angles}
\usepackage{caption}
\usepackage[export]{adjustbox}

\usepackage{paralist}

\newtheorem{thmUniv}{Theorem}

\newtheorem{theorem}{Theorem}[section]
\newtheorem{corollary}[theorem]{Corollary}
\newtheorem{proposition}[theorem]{Proposition}
\newtheorem{lemma}[theorem]{Lemma}
\newtheorem{conjecture}[theorem]{Conjecture}
\newtheorem*{theorem*}{Theorem}%

\theoremstyle{definition}
\newtheorem{definition}[theorem]{Definition}
\newtheorem{example}[theorem]{Example}
\newtheorem{remark}[theorem]{Remark}

\AddToHook{env/theorem/begin}{\crefalias{section}{theorem}}
\AddToHook{env/proposition/begin}{\crefalias{theorem}{proposition}}
\AddToHook{env/corollary/begin}{\crefalias{theorem}{corollary}}
\AddToHook{env/lemma/begin}{\crefalias{theorem}{lemma}}
\AddToHook{env/conjecture/begin}{\crefalias{theorem}{conjecture}}
\AddToHook{env/definition/begin}{\crefalias{theorem}{definition}}
\AddToHook{env/example/begin}{\crefalias{theorem}{example}}
\AddToHook{env/remark/begin}{\crefalias{theorem}{remark}}
\AddToHook{env/observation/begin}{\crefalias{theorem}{observation}}
\AddToHook{env/question/begin}{\crefalias{theorem}{question}}
\AddToHook{env/notation/begin}{\crefalias{theorem}{notation}}
\AddToHook{env/assumption/begin}{\crefalias{theorem}{assumption}}
\AddToHook{env/convention/begin}{\crefalias{theorem}{convention}}

\crefname{equation}{Equation}{Equations}

\newcommand{\R}{\mathbb{R}} %
\newcommand{\N}{\mathbb{N}} %
\newcommand{\Z}{\mathbb{Z}} %
\renewcommand{\c}[1]{{\mathcal{#1}}} %
\renewcommand{\b}[1]{{\boldsymbol{#1}}} %
\renewcommand{\epsilon}{\varepsilon} %
\newcommand{\polytope}[1]{\mathsf{#1}}
\newcommand\pol[1][P]{\polytope{#1}}
\newcommand{\polQ}{\pol[Q]}
\newcommand{\polS}{\pol[S]}
\newcommand{\polF}{\pol[F]}
\newcommand{\polD}{\pol[D]}
\newcommand{\Chim}{\pol[Chim]}

\newcommand{\set}[2]{\left\{ #1 \;\middle|\; #2 \right\}} %
\newcommand{\ssm}{\smallsetminus} %
\newcommand{\coloneqq}{\mbox{\,\raisebox{0.2ex}{\scriptsize\ensuremath{\mathrm:}}\ensuremath{=}\,}} %
\newcommand{\simplex}{\polytope{\Delta}} %
\DeclareMathOperator{\conv}{conv} %
\DeclareMathOperator{\nVol}{nVol}
\DeclareMathOperator{\vol}{Vol}
\DeclareMathOperator{\Ehr}{Ehr}
\DeclareMathOperator{\UniEq}{\simeq_{\Z}}

\newcommand{\ie}{\textit{i.e.,}~} %
\newcommand{\eg}{\textit{e.g.,}~} %
\newcommand{\aka}{\textit{a.k.a.}~} %
\definecolor{darkblue}{rgb}{0,0,0.7} %
\definecolor{green}{RGB}{57,181,74} %
\definecolor{violet}{RGB}{147,39,143} %
\newcommand{\darkblue}{\color{darkblue}} %
\newcommand{\defn}[1]{\textsl{\darkblue #1}} %
\newcommand{\mathdefn}[1]{{\darkblue #1}} %

\newcommand{\OEIS}[1]{\href{https://oeis.org/#1}{\cite[#1]{OEIS}}} %

\newcommandx{\Sylv}[2][1=k,2=d]{
\IfInteger{#1}
{\ifnum #1 < 0 {\pol[Base]_{#2}} \else {\pol[Sylv]^{#1}_{#2}}\fi}
{\pol[Sylv]^{#1}_{#2}}
}
\newcommandx{\Sk}[2][1=k,2=d]{\Sylv[#1][#2]}
\newcommandx{\Base}[1][1=d]{\Sylv[-1][#1]}
\newcommandx{\SimplSet}[2][1=d,2=k]{\c S(d, k)}
\newcommandx{\EhrPoly}[2][1=P,2=t]{\Ehr(#1;#2)}
\newcommandx{\hstarPoly}[2][1=P,2=t]{h^{\ast}(#1;#2)}

\makeatletter
\newcommand{\sectionnotoc}[1]{%
  {\let\addcontentsline\@gobblethree
   \section*{#1}}%
}

\usepackage{todonotes}

\thanks{}

\AtEndDocument{%
\bigskip\bigskip\bigskip\bigskip\bigskip
\footnotesize
\noindent
\begin{minipage}[t]{0.65\textwidth}
\textsc{Universität Osnabrück, Germany}\\
Martina Juhnke: \texttt{martina.juhnke@uni-osnabrueck.de}\\
Germain Poullot: \texttt{germain.poullot@uni-osnabrueck.de} \\
Jhon B. Caicedo: \texttt{jhon.bladimir.caicedo.portilla@uni-osnabrueck.de}
\end{minipage}
\hspace{-1.25cm}
\begin{minipage}[t]{0.42\textwidth}
\textsc{Centro de Investigaci\'on en Matem\'aticas, Guanajuato}\\
Federico Castillo: \texttt{federico.castillo@cimat.mx}
\end{minipage}
}

\begin{document}

\title{Sylvester simplices:\linebreak
Triangulations and Ehrhart-theoretic aspects}
\author{Jhon B. Caicedo, Federico Castillo, Martina Juhnke and Germain Poullot}
\date{}

\begin{abstract}
The Sylvester simplex $\Sylv$ is a $d$-dimensional lattice simplex with exactly $k$ interior lattice points.
Sylvester simplices are conjectured to be the volume maximizers among all $d$-dimensional lattice polytopes with exactly $k$ interior lattice points for any $k\geq 1$.
Even stronger, it is conjectured that they maximize (entry-wise) the $h^\ast$-vector among all $d$-dimensional lattice polytopes with exactly $k$ interior lattice points.
Yet, Sylvester simplices seem to be rarely studied in their own right.
In particular, their Ehrhart-theoretic properties are far from being well understood.

In the present article, we tackle this problem.
We describe flag, regular and unimodular triangulations for the Sylvester simplices, and prove that their $h^\ast$-vectors are unimodal.
Moreover, we explicitly determine the values of some entries of their $f^\ast$-vectors, and prove that they are Ehrhart magic positive up to dimension $6$ but not in dimension $7$.
We conclude by detailing tables of Ehrhart-theoretic quantities (numbers of lattice points, Ehrhart polynomials, local and boundary $h^\ast$-vectors, $f^\ast$-vectors) for Sylvester simplices of dimensions 7 and lower.
\end{abstract}

\maketitle

\vspace{-0.5cm}
\tableofcontents

\vspace{-0.5cm}
\sectionnotoc{Acknowledgements}

Jhon B. Caicedo was supported by DFG grant JU 3097/4-1.
The authors deeply thank Gabriele Balletti, Matthias Beck, Christian Haase and  Benjamin Nill for insightful conversations.
This project started during the Intensive Research Program on Combinatorial Geometries and Geometric Combinatorics,
held at the Centre de Recerca Matemàtica in October-November 2025, and funded by the Severo Ochoa and María de Maeztu Program for Centers and Units of Excellence in R\&D (CEX2020-001084-M), the Institute de Matemàtica de la Universitat de Barcelona, and the Spanish projects PID2022-137283NB-C21 and RED2024-153572-T of MICIU/AEI /10.13039/501100011033.\\

\sectionnotoc{Use of AI declaration}

We used AI models (namely ChatGPT using GPT-5.5 then GPT-5.6 Sol, and Gemini 3.1) to:
\begin{itemize}
\item Correct mistakes in English writing (\ie not to write for us),
\item Search the literature when we were not sure of the first reference for a result (\ie not to summarize articles);
\item Improve the readability of our tables in Appendix~\ref{app:Tables} (\ie not to compute their content);
\item After having identified all the patterns that were visible to our human eyes, we also asked Gemini 3.1 to search for numerical coincidences in \Cref{tab:f_star_coefficients}, and it hinted at the values of $x_{d, d-1}$ and $y_{d, d-1}$ presented in \Cref{prop:fStarExplicitValues}: we then designed a better expression and proved it ourselves;
\item Improve our implementation of the computation (\ie not to guess the values) of $h^\ast(\Sylv; t)$ derived from \Cref{thm:hStarSylvester} in order to compute the case $d = 7$ and $k = 0$: this numerical experiment led to \Cref{prop:NotMagicPositive}.
\end{itemize}

Apart from these specific points, all statements, proofs, concepts, notations, figures, references to the literature, etc. were made by us, as humans, using non-AI-powered tools.
Computations of the Ehrhart-theoretic quantities in Appendix~\ref{app:Tables} were run using SageMath~\cite{Sage} and our own  code (except for $d = 7$), not relying on AI-powered computations.\\

\section{Introduction}\label{sec:Intro}

Let $\mathdefn{[n]} \coloneqq \{1, \dots, n\}$ for any positive integer $n$.
Fix integers $d\geq 1$ and $k\geq 0$.
We denote by $\mathdefn{\b e_1}, \dots, \mathdefn{\b e_d}$ the standard basis vectors of $\R^d$, and by $\mathdefn{\b 0} \coloneqq (0, \ldots, 0)$ the zero-vector of $\R^d$.
A \defn{lattice polytope} is a polytope whose vertices lie in $\Z^d$ (\ie are \defn{lattice points}).

We denote by $\mathdefn{\mathscr{P}_d(k)}$ the family of all $d$-dimensional lattice polytopes in $\R^d$ with exactly $k$ interior lattice points, and by $\mathdefn{\mathscr{S}_d(k)}$ the family of all simplices in $\mathscr{P}_d(k)$.
It is known that, for every fixed $k \geq 1$, the volume of a polytope in $\mathscr{P}_d(k)$ is bounded from above (see \cite{hensley1983lattice}), and in recent years many researchers have worked to determine optimal bounds, \eg \cite{averkov2015largest, averkov2020lattice, averkov2020local, Balletti_dual_volume}.
Note that, for $k = 0$, the volume of the simplices in $\mathscr S_d(0)$ (and hence of the polytopes in $\mathscr P_d(0)$) is not bounded, since the famous Reeve simplices \cite{Reeve1957} have no interior lattice points but can have arbitrarily large volume.

In \cite{wills1982lattice}, Wills, Zaks and Perles define the $d$-dimensional simplex
\begin{equation*}\label{eq:sylvester_definition}
\mathdefn{\Sylv} \coloneqq \conv (\b 0,\, s_1\b e_1,\, \ldots,\, s_{d-1}\b e_{d-1},\, (k+1)(s_d - 1)\b e_d), \text{ with }  s_{0}=1, s_{i}= 1+ \prod_{j=0}^{i-1}s_j, \text{ for } i\geq 1.
\end{equation*}
The sequence $ ( s_n )_{n\in \mathbb{N}} $ is called  \defn{Sylvester sequence}.
Its first terms are $1, 2, 3, 7, 43, 1807$, see \OEIS{A000058}.
In this work, we will refer to the simplices $\Sylv$ as \defn{Sylvester simplices}: the same simplices also appear under the names \defn{Zaks--Perles--Wills simplices} \cite{wills1982lattice,averkov2020local},  \defn{Hensley simplices} \cite{hensley1983lattice,averkov2015largest}, or \defn{Lagarias--Ziegler simplices} \cite{lagarias1991bounds}. %
In addition, the facet of the Sylvester simplex opposite to the vertex $\b e_d$ is the \defn{Sylvester base}, \ie $\mathdefn{\Sylv[-1]} \coloneqq  \conv (\b 0,\, s_1\b e_1,\, \ldots,\, s_{d-1}\b e_{d-1})$.

The set of interior lattice points of the simplex $\Sylv$ is $ \left\{ (1,\dots,1,i) \mid i \in [k] \right\} $, see \Cref{prop:SylvesterSimplexBasicFacts}~\eqref{item:SylvInteriorLatticePoints}.
Its volume is given by:

\begin{equation*}
\vol(\Sylv) = \dfrac{1}{d!}(k+1)(s_d-1)\prod_{j<d} s_j = \dfrac{1}{d!}(k+1)(s_d-1)^2.
\end{equation*}

The connection between $\Sylv$ and the bounded volume for $\mathscr{S}_d(k)$ and $\mathscr{P}_d(k)$ is given by the following conjecture of Balletti and Kasprzyk.

\begin{conjecture}[{\cite[Conjecture~1.5]{balletti2016three}}]\label{conj:volume}
Let $d\geq 3$ and $k\geq 1$.
Then for every $\pol\in\mathscr{P}_d(k)$
$$ \vol(\pol) \leq \vol(\Sylv).    
$$
Moreover, $\Sylv$ is the \emph{only} polytope achieving this upper bound except when $d = 3$ and $k = 1$. 
\end{conjecture}

For $k=1$, this conjecture has been proven when $\pol$ is a simplex by Averkov, Kr\"umpelmann and Nill \cite[Theorem~2.2]{averkov2015largest}, and when $\pol$ is reflexive by Balletti, Kasprzyk and Nill \cite[Corollary~1.2]{Balletti_dual_volume}.
Moreover, this conjecture holds true for arbitrary $3$-dimensional lattice polytopes by Kasprzyk’s classification of canonical Fano threefolds \cite[Theorem~4.2]{Kasprzyk} (see also \cite[Section~1.1]{Balletti_dual_volume}).
For arbitrary lattice polytopes in dimensions $d\geq 4$, the conjecture remains open.
Balletti, Kasprzyk and Nill proved the corresponding sharp upper bound for the volume of the dual polytope. 
Moreover, Averkov, Kr\"umpelmann and Nill proved a nearly optimal bound on the volume of simplices in $\mathscr S_d(k)$ in \cite[Theorem~1.2]{averkov2020lattice}, namely $(d+1)\cdot \vol(\Sylv)$.

\medskip

Generally, research around the Sylvester simplex has focused on its volume.
Yet, Balletti \cite[Section~8]{balletti2021enumeration} showed by enumerating all $3$-dimensional simplices with at most $11$ interior lattice points that other Ehrhart-theoretic quantities attain their maxima over $\mathscr S_3(k)$ or $\mathscr P_3(k)$ exactly for the Sylvester simplex $\Sylv[k][3]$, if $1\leq k\leq 11$.

\begin{conjecture}\label{conj:hstar}
For $d\geq 3$ and $k\geq 1$, all $\pol\in\mathscr{P}_d(k)$ satisfy
$$h^\ast_i(\pol) \leq h^\ast_i(\Sylv) ~~ \text{ for all }~~ i=1,\dots,d.$$
Moreover, $\Sylv$ is the \emph{only} polytope achieving these upper bounds at once except if $d = 3$, $k = 1$.
\end{conjecture}

The $3$-dimensional case appears in \cite[Conjecture~6.1]{balletti2016three} and in a refined version in \cite[Conjecture~8.7]{balletti2021enumeration}.
Note that \Cref{conj:hstar} implies \Cref{conj:volume}, since $\vol(\pol) = \frac{1}{d!}\sum_i h^\ast_i(\pol)$.

It is worth commenting on the case $d=2$, which, following \cite{balletti2016three}, we exclude from the statement of the conjecture.
In dimension $2$, Scott's classification \cite{scott1976convex} of possible $h^\ast$-vectors implies that $h^\ast_1 \le 3h^\ast_2 + 3$ whenever $h^\ast_2 \geq 2$.
The Sylvester triangle $\conv\{(0,0), (2,0), (0,2\cdot (k+1))\}$ satisfies $h^\ast_2=k$ and $h^\ast_1=3k+3$, so for $k\ge2$ it attains this upper bound with equality.
Two caveats are worth pointing out.
First, the maximizer is not unique: the rectangle with vertices $(0,0)$, $(2,0)$, $(2,k+1)$, $(0,k+1)$ has the same $h^\ast$-vector, $h^\ast_2=k$ and $h^\ast_1=3k+3$.
Second, the case $h^\ast_2=1$ behaves differently: the true maximum for $h^\ast_1$ there is $7$, not $6$, and it is achieved by the triangle with vertices $(0,0)$, $(3,0)$, $(0,3)$.

But there may be a great clue we can extract from dimension $2$.
Haase and Schicho \cite{haase2009lattice} show that whenever $h^\ast_1$ is maximal for a given $h^\ast_2$, the interior lattice points must all lie on a single line.
The same property holds for every Sylvester simplex in higher dimensions, see \Cref{prop:SylvesterSimplexBasicFacts}~\eqref{item:SylvInteriorLatticePoints}, and it may be the key to the general pattern.

\medskip

After presenting the very basic properties of the Sylvester simplex in \Cref{sec:PrelimAndBasicFacts}, we will present two ways of triangulating the Sylvester simplices in \Cref{sec:DissectingSylvesterSimplex}.
Notably, we prove:

\begin{thmUniv}[{\Cref{thm:fruTriangulationSylvesterSimplices}}]\label{thmA}
The Sylvester simplex $\Sylv$ admits a flag, regular and unimodular triangulation for all $d \geq 1$ and $k \geq 0$.
\end{thmUniv}

This implies that the simplex $\Sylv$ has the \emph{integer decomposition property} (IDP), among others.
Since the triangulation from \Cref{thmA} is not really explicit (as it relies on Chimney polytopes and one-point extensions),  we propose a second (non-unimodular) triangulation of $\Sylv$ into unimodularly equivalent simplices (\Cref{lem:UnimodularEquivalenceLayering}). We refer to this second triangulation as the \defn{layering triangulation} of $\Sylv$.
As a consequence we obtain the following result.

\begin{thmUniv}[{\Cref{thm:HaveAffineCoefficients}}]\label{thmB}
The coefficients $h^\ast_i(\Sylv)$ are affine functions in $k$.
\end{thmUniv}

The same holds for the main Ehrhart-theoretic quantities, \eg the number of lattice points of~$\Sylv$, the coefficients of its $f^\ast$-vector and its Ehrhart polynomial (whether expressed in the magic basis or not) and their counterparts for local or boundary $h^\ast$- and $f^\ast$-vectors.
This affine structure allows us to determine the $h^\ast$-vector of $\Sylv$ for all $k\geq 0$ just by computing it for $k \in \{0,1\}$, which we present in \Cref{ssec:example_small_dimension} and in \Cref{tab:LatticePointsNbr,tab:ehrhart_coefficients,tab:magic_coefficients,tab:hstar_coefficients,tab:local_h_coefficients,tab:boundary_h_coefficients,tab:f_star_coefficients} for dimensions $d \leq 6$, and in Appendix \ref{ssec:TablesD7} for $d = 7$.

In \Cref{sec:fStarPolynomialsSylvesterSimplices}, we turn our attention to the $f^\ast$-vectors of $\Sylv$, that is, counting the number of simplices of each dimension in a unimodular triangulation of $\Sylv$.
We show that the $f^\ast$-vectors of $\Sylv$ satisfy a surprising recursive formula on $d$ (\Cref{thm:f_star_induction}), and extract the number of simplices of dimensions $d$, $d-1$ and $d-2$ of any unimodular triangulation of $\Sylv$ (\Cref{prop:fStarExplicitValues}).

Last but not least, in \Cref{sec:Computing_hStar}, we detail another way of computing the $h^\ast$-polynomial of $\Sylv$ (\Cref{thm:hStarSylvester}), relying on the fact that it is a \emph{diagonal simplex}, \ie its vertices are $\b 0$ and points lying on the coordinate axes.
We derive similar expressions for the local $h^\ast$-vectors and the boundary $h^\ast$-vectors (\Cref{prop:local_h_star_symmetry,prop:boundaryhStarSylvester}).
Then, using Ehrhart--McDonald reciprocity, we prove an unusual symmetry for the entries of the $h^\ast$-vector of $\Sylv$ (\Cref{prop:SymmetryhStar}), showing:

\begin{thmUniv}[{\Cref{thm:hStarUnimodal}}]\label{thmC}
For $d \geq 1$ and $k \geq 0$, the coefficients $h_i^\ast(\Sylv)$ form a unimodal sequence with mode at $\frac{d}{2}$ if $d$ is even, and with mode at $\frac{d-1}{2}$ or $\frac{d+1}{2}$ if $d$ is odd.
\end{thmUniv} 

We conclude by discussing magic positivity (which implies both Ehrhart positivity and $h^\ast$-real-rootedness, see \cite[Figure 1]{FerroniHigashitani2024ExamplesCounterexamplesEhrhartTheory}) in \Cref{sec:MagicPositivity}.
We prove:

\begin{thmUniv}[{\Cref{coro:magic_positive,prop:NotMagicPositive}}]\label{thmD}
 For $1\leq d \leq 6$ and all $k \geq 0$ the Ehrhart polynomial $\Ehr(\Sylv; t)$ is magic positive. Hence, $\Sylv$ is Ehrhart positive and  $h^\ast(\Sylv; t)$ is real-rooted.

For any $k\geq 0$, the linear coefficient of the  Ehrhart polynomial $\Ehr(\Sylv[k][7]; t)$ is negative. In particular, $\Sylv[k][7]$ is  not magic positive.
\end{thmUniv}

To motivate future research, we conjecture (\Cref{conj:MagicPositive}) that $\Ehr(\Sylv[k]; t)$ has a negative  coefficient for all $d \geq 7$ and hence is not magic positive for $d \geq 7$.
On the other hand, we also conjecture (\Cref{conj:hStarRealRooted}) that the $h^\ast$-polynomial of $\Sylv$ is real-rooted for all $d\geq 1$ and $k\geq 0$ which we proved for $d\leq 6$.

\section{Preliminaries}\label{sec:PrelimAndBasicFacts}

A sequence of non-negative integers $\b a = (a_1, \dots, a_s)$ is \defn{unimodal} if there exists $m\in[s]$, called the \defn{mode} of $\b a$, such that $(a_1, \dots, a_m)$ is  weakly increasing, while $(a_m, \dots, a_s)$ is weakly decreasing.
A notorious property is that if the polynomial $\sum_{i=1}^s a_i t^i$ is \defn{real-rooted} (\ie does not have any complex non-real root) and has non-negative coefficients, then $\b a$ is unimodal.
We refer to Br\"and\'en's handbook \cite{Branden2015-UnimodalityLogConcavityRealRootednessGammaPositivity} for all the ``well-known'' facts on unimodality and real-rootedness, and to \cite[Figure~1]{FerroniHigashitani2024ExamplesCounterexamplesEhrhartTheory} for a particularly pedagogical illustration of the importance of unimodality and real-rootedness in the study of Ehrhart-theoretic quantities.

\subsection{Lattice polytopes}

We now recall the usual key notions concerning lattice polytopes.

The \defn{normalized volume} of a full-dimensional polytope $\pol$ is $\mathdefn{\nVol(\pol)} \coloneqq (\dim\pol)! \cdot \vol(\pol)$ where $\vol(\pol)$ is its Euclidean volume.

A lattice polytope $\pol = \conv(\b v_1, \dots, \b v_n)$ is \defn{reflexive} if it has a unique interior lattice point~$\b p$, and if the polar polytope $(\pol-\b p)^{\triangle} \coloneqq \set{\b x\in \R^d}{\langle \b x, \b v_i - \b p\rangle \leq 1 \text{ for } i\in [n]}$ is a lattice polytope.

The \defn{primitive vector} associated to a vector $\b x = (x_1, \dots, x_d)\neq \b 0$ is the vector $\frac{1}{c}\b x$, where $ c= \gcd(|x_1|, \dots, |x_d|)$.
A polytope $\pol\subseteq\R^d$ is \defn{smooth} if, at each vertex, the primitive vectors of its adjacent edges form a $\Z$-basis of $\Z^d$.

Since we are working with lattice points and lattice polytopes, most of the quantities we care about are invariant under an automorphism of the lattice $ \mathbb{Z}^{d} $.
To make this more precise we denote a map $\varphi:\R^d\to\R^d$ of the form $\varphi(\b x)=U\b x+\b b$, where $U\in\Z^{d\times d}$ satisfies $\det(U)=\pm 1$ and $\b b\in\Z^d$, an \defn{affine unimodular transformation}. 
Two lattice polytopes $\pol,\polytope{Q}\subseteq\R^d$ are \defn{unimodularly equivalent}, denoted $\mathdefn{\pol \UniEq \polQ}$, if $\polytope{Q}=\varphi(\pol)$ for some affine unimodular transformation $\varphi$.
Since a unimodular transformation induces a bijection between the lattice points of its input and its output, it follows that if 
$\pol \UniEq \polQ$, then $|\pol \cap \mathbb{Z}^{d}| = |\polytope{Q} \cap \mathbb{Z}^{d}|$.

This invariance becomes particularly useful when a lattice polytope is decomposed into simpler pieces, leading naturally to the study of triangulations.
\begin{definition}
A \defn{triangulation} of a $d$-polytope $\pol$ is a finite collection $\c T$ of $d$-simplices satisfying
\begin{itemize}
\item $\pol = \bigcup_{\simplex\in \c T}\simplex$, and
\item for all $\simplex_i,\simplex_j \in \c T$, the intersection $\simplex_i \cap \simplex_j$ is a face of both $\simplex_i$ and $\simplex_j$.
\end{itemize}

Let $V(\c T)$ be the collection of all the vertices of the simplices of $\c T$.
We say that a triangulation~$\c T$ is \defn{unimodular} if each $d$-simplex has normalized volume equal to~$1$;
\defn{flag} if $\c T$ is completely determined by its $1$-skeleton (\ie if for $X\subseteq V(\c T)$ if vertices in $X$ are pairwise linked by edges of simplices of $\c T$, then $\conv(X)$ is a face of some simplex of $\c T$);
and \defn{regular} if there exists a function\linebreak $\omega:V(\c T)\to\R$ such that the maximal simplices of $\c T$ are precisely the projections of the lower facets of $\conv\bigl((\b v,\omega(\b v))\mid \b v\in V(\c T)\bigr)\subseteq\R^{d+1}$ onto $\R^d$.

Flag, regular and unimodular triangulations are sometimes called \emph{quadratic triangulation}, see for instance \cite[Figure~1]{FerroniHigashitani2024ExamplesCounterexamplesEhrhartTheory} for additional background on this type of triangulations.
\end{definition}

\subsection{Ehrhart polynomials, \texorpdfstring{$h^\ast$}{h*}-polynomials and \texorpdfstring{$f^\ast$}{f*}-polynomials}\label{subs:fandh}
\label{ssec:ehrhart}

For a lattice polytope $\pol\subseteq \R^d$, let $\mathdefn{\Ehr(\pol; m)} \coloneqq \bigl| m\pol \cap \mathbb{Z}^d \bigr|$, where $\mathdefn{m\pol} \coloneqq \set{m\cdot \b p}{\b p \in \pol}$ denotes the $m\textsuperscript{th}$ dilation of $\pol$ for an integer $m\geq 0$.
Then $\Ehr(\pol; m)$ agrees with a polynomial function in $m$ of degree $\dim \pol$ with rational coefficients, called  \defn{Ehrhart polynomial} of $\pol$, see \cite{Ehrhart_poly}.
Furthermore, if two lattice polytopes are unimodularly equivalent, then they have the same Ehrhart polynomial.

\begin{definition}
The \defn{Ehrhart series} of a $d$-dimensional lattice polytope $ \pol $ is the generating function
\begin{equation*}\label{eq:ehrhart_series}
	\sum_{m\geq 0} \Ehr(\pol;m)\cdot t^m =\frac{h^{\ast}_0(\pol) + h^{\ast}_1(\pol) \cdot t + \cdots + h^{\ast}_d(\pol) \cdot t^d}{(1-t)^{d+1}} \,,
\end{equation*}
where $\mathdefn{h^\ast(\pol; t)} = h^{\ast}_0(\pol) + h^{\ast}_1(\pol) \cdot t + \cdots + h^{\ast}_d(\pol) \cdot t^d$ is called the \defn{$h^{\ast}$-polynomial}, and its vector of  coefficients $h^\ast(\pol)\coloneqq \bigl(h^\ast_0(\pol), \ldots, h^\ast_d(\pol)\bigr)$ is called the \defn{$h^{\ast}$-vector} of $\pol$.
\end{definition}

The $ h^{\ast}$-polynomial carries the same information as the Ehrhart polynomial, but turns out to have more amenable properties (see \cite[Sections~3.5 \& 3.6]{Beck_book}).
For every lattice polytope $\pol$:
\begin{itemize}
\item for each $ i \in [d] $, the coefficient $h^\ast_i(\pol)$ is a non-negative integer \cite{stanley1980decompositions}, and
\item we have $\nVol(\pol) = \sum_{i=0}^d h^\ast_i(\pol)$, and
\item the first entries are  $h^\ast_0(\pol) = 1$ and $h^\ast_1(\pol) = |\pol \cap \mathbb{Z}^d| - d - 1$, and
\item the last entry is the number of interior lattice points of $\pol$, \ie $ h^\ast_d(\pol) = |\pol^\circ \cap \mathbb{Z}^d|$.
\end{itemize}

When $\pol[S] = \conv(\b v_0, \ldots, \b v_d)\subseteq \mathbb{R}^d$ is a lattice $d$-simplex, its $h^\ast$-vector admits a beautiful geometric interpretation.
Indeed, let its associated \defn{half-open fundamental parallelepiped} be:
\begin{equation}\label{eq:half_open_para}
	\mathdefn{\Pi^\circ_{\polS}} \coloneqq \set{\sum_{j=0}^d \lambda_j \begin{pmatrix} \b v_j \\ 1\end{pmatrix}}{0\leq \lambda_j < 1 \text{ for } 0\leq j\leq d} \,.
\end{equation}
It is well-known, see \cite[Corollary 3.11]{Beck_book}, that we have $h^\ast_i(\polS) = \left|\set{\b x\in\Pi^\circ_{\polS}\cap\mathbb{Z}^{d+1}}{x_{d+1}=i}\right|$, \ie the coefficient $h^\ast_i(\polS)$ is the number of lattice points in $\Pi^\circ_{\polS}$ whose last coordinate is $i$.
Following Breuer \cite{Felix_f_star}, the \defn{$f^\ast$-vector} of a $d$-dimensional lattice
polytope $\pol$ is the vector $\bigl(f_0^\ast(\pol),\ldots,f_d^\ast(\pol)\bigr)$ defined by
\[
  \Ehr(\pol;n)=\sum_{i=0}^d f_i^\ast(\pol)\binom{n-1}{i} \,.
\]
We also set $f_{-1}^\ast(\pol)=1$.  The $f^\ast$- and $h^\ast$-vector carry the same information and one can recover one from the other as follows (with the convention $h^\ast_{d+1}(\pol) = 0$):

\begin{align*}
  f_k^\ast(\pol)
  &=\sum_{j=0}^{k+1}\binom{d-j+1}{k-j+1}h_j^\ast(\pol)
  &&\text{for } 0\le k\le d,%
  \\
  h_k^\ast(\pol)
  &=\sum_{j=-1}^{k-1}(-1)^{k-j-1}
    \binom{d-j}{k-j-1}f_j^\ast(\pol)
  &&\text{for } 0\le k\le d.%
\end{align*}
Equivalently, 
\[
  (1+t)^{d+1}\cdot
  h^\ast\!\left(\pol;\frac{t}{1+t}\right)=\sum_{i=-1}^d f_i^\ast(\pol) \, t^{i+1}
\]
and the polynomial $f^\ast(\pol;t) = \sum_{i=-1}^d f_i^\ast(\pol)\,t^{i+1}$ is called \defn{$f^\ast$-polynomial} of $\pol$.

\newpage
Breuer proved that all $f_i^\ast(\pol)$ are non-negative integers \cite{Felix_f_star}.
Some further properties of the $f^\ast$-vectors are
\begin{itemize}
\item The function $ f^{\ast}_i(\pol) $ is a \defn{valuation} on polytopes: that is if $\pol\cup\polQ$ is a polytope, then $f^\ast_i(\pol\cup\polQ) + f^\ast_i(\pol\cap\polQ) = f^\ast_i(\pol) + f^\ast_i(\polQ)$ for any $-1\leq i\leq d$.
\item The first entries are  $f^\ast_{-1}(\pol) = 1$ and $f^\ast_0(\pol) = |\pol \cap \mathbb{Z}^d|$.
\item The last entry is $ f^\ast_d(\pol) = \nVol(\pol)$.
\end{itemize}

The valuation property is key as it will allow us to compute $ f^\ast(\pol; t)$ from a (unimodular) triangulation.

\begin{remark}\label{rmk:BekteMcMullen}
The existence of a unimodular triangulation has important consequences in the study of lattice polytopes.
By \cite[Theorem 1]{BetkeMcMullen1985-LatticePoints}, if $\c T$ is a unimodular triangulation of a lattice $d$-polytope $\pol$, then the $h^{\ast}$- and $f^{\ast}$-vectors coincide with the $h$- and $f$-vectors of the triangulation.
In particular, $f^\ast_i(\pol)$ is the number of $i$-dimensional simplices of any unimodular triangulation of~$\pol$.
\end{remark}

\begin{example}
Consider the $3$-dimensional cube $\pol[Cube]_3 = [0,1]^3$.
This cube has a unimodular triangulation:
$\c T = \left\{ \conv(\b 0, \b e_{\pi(1)}, \b e_{\pi(1)} + \b e_{\pi(2)}, \b e_{\pi(1)} + \b e_{\pi(2)} + \b e_{\pi(3)}) \mid \pi:[3]\to[3] \text{ a bijection}\right\}$.
By counting faces of $\c T$, we get that the $ f^{\ast}$-vector of $\pol[Cube]_3$ is $ (1,8,19,18,6)$.
For example, $ f^{\ast}_{3}(\pol[Cube]_3) = 6$ as there are six simplices in the triangulation $\c T$.
Since $\c T$ is unimodular, $6$ is also the normalized volume of $\pol[Cube]_3$.

The relation between $h^\ast(\pol[Cube]_3)$ and $f^\ast(\pol[Cube]_3)$ allows to determine the $ h^{\ast}$-vector of $\pol[Cube]_3$, namely $(1,4,1,0)$.
Notice that $ h^{\ast}_{3}(\pol[Cube]_3) = 0$ as $\pol[Cube]_3$ has no interior lattice points.
\end{example}
We refer to \cite{BetkeMcMullen1985-LatticePoints,Felix_f_star,BeckDeligeorgakiHlavacekValenciaPorras2024-fStarVectors} for a more detailed study of the $f^\ast$-polynomial.

\begin{remark}
Since the $f^\ast$-vector is a non-negative combination of the $h^\ast$-vector, it follows that \Cref{conj:hstar} would imply an analogous maximality for $f^\ast$-vectors among polytopes with a fixed number of interior points.
\end{remark}

In fact, from the expansion of $f^\ast$ in terms of $h^\ast$ it is not hard to establish an analogue of \Cref{conj:hstar}; that is, for polytopes with fixed volume, rather than interior points.

\begin{theorem}\label{thm:max_fstar}
Let $d \ge 1$ and $K \ge 1$ be integers, and let $\pol \subseteq \mathbb{R}^n$ be a lattice $d$-polytope with $f^\ast_d(\pol) = K$.
Then for every $-1 \le i \le d$ we have
\[
f^\ast_i(\pol) \leq f^\ast_i(\simplex_{d,K}) = \binom{d+1}{i+1} + (K-1)\cdot \binom{d}{i},
\]
where $\simplex_{d,K} \coloneqq \conv(\b 0, \b e_1, \dots, \b e_{d-1}, K \b e_d)$.
\end{theorem}

\begin{proof}
Since $h^\ast_0(\pol) + \dots + h^\ast_d(\pol) = K$ and they are non-negative with $h^\ast_0(\pol) = 1$, we have that $h^\ast_1(\pol) + h^\ast_2(\pol) + \dots + h^\ast_d(\pol) \leq K - 1$.
Since $\binom{d}{i} \geq \binom{d-j+1}{i-j+1} $ for $j \geq 1$, we get for $0 \leq i \leq d-1$
\[
\begin{array}{rcl}
f_i^\ast(\pol)
&=& \sum_{j=0}^{i+1} \binom{d-j+1}{i-j+1}h_j^\ast(\pol) = \binom{d+1}{i+1} h^\ast_0(\pol) ~+~ \sum_{j=1}^{i+1} \binom{d-j+1}{i-j+1}h_j^\ast(\pol) \\
&\leq& \binom{d+1}{i+1} h^\ast_0(\pol) + \binom{d}{i} \bigl(h^\ast_1(\pol) + h^\ast_2(\pol) + \dots + h^\ast_d(\pol)\bigr) \\
&=& \binom{d+1}{i+1} + \binom{d}{i}\cdot (K-1)
\end{array}
 \]

Equality is achieved if and only if the only non-zero term in the sum $\sum_{j=1}^{i+1} \binom{d-j+1}{i-j+1}h_j^\ast(\pol)$ is the one for $j = 1$, that is if and only if $h^\ast_2(\pol) = \dots = h^\ast_d(\pol) = 0$ and $h^\ast_1(\pol) = K-1$.

It is straightforward to check that $\simplex_{d,K}$ has normalized volume $K$ and $d+K$ lattice points, hence its $h^\ast$-vector is indeed $(1, K-1, 0, \dots, 0)$.
\end{proof}

The maximizer $\simplex_{d, K}$ is not unique:
any polytope with the same $h^\ast$-vector is a maximizer too.
Batyrev and Nill \cite[Theorem 2.5]{batyrev2006multiples} classified all such polytopes, and curiously they are all simplices.
This is another meta-evidence that simplices are ``extremal'' polytopes.

\subsection{Sylvester simplices}

\medskip

We quickly review some easy properties of the Sylvester simplex.
We do not claim that these results are original (or difficult), but rather wish to gather these basic facts in a single place.

\newpage
We will first need the following straightforward properties of the Sylvester sequence:
\begin{compactitem}
\item the first terms are $1, 2, 3, 7, 43, 1\,807, 3\,263\,443$, see \OEIS{A000058} (where $s _0 = 1$),
\item for all $0\leq i < d$, the number $\frac{s_d-1}{s_i} = \prod_{j< d,\, j\ne i} s_j$ is a positive integer,
\item for all $d\geq 1$, we have: $\tfrac{1}{s_d-1} + \sum_{1\leq i < d} \tfrac{1}{s_i} = 1$.
\end{compactitem}

As the Sylvester simplex is not the only one to exhibit these properties, we define the following class of simplices.

\begin{definition}
A \defn{diagonal $d$-simplex} is a simplex $\simplex\subseteq\R^d$ of the form $\simplex = \conv(\b 0, a_1\b e_1, \dots, a_d\b e_d)$ for integers $d \geq 2$ and $0 < a_1 \leq \dots \leq a_d$.
A diagonal $d$-simplex is \defn{$k$-Egyptian} for $k \geq 0$ if:
\begin{equation}\label{eq:k_egyptian}
	(k+1) \text{ divides } a_d \,, \qquad a_{d-1} \leq \frac{a_d}{k+1} \,, \qquad \text{ and } \qquad \frac{k+1}{a_d} ~+~ \sum_{i=1}^{d-1} \frac{1}{a_i} = 1\,.
\end{equation}

For $d = 1$, by convention, we consider $[0, k+1]$ to be the only $k$-Egyptian diagonal $1$-simplex.
\end{definition}

The name ``Egyptian'' comes from the notion of \emph{Egyptian fractions} which are fractions of the form $\frac{1}{m}$ for some integer $m \geq 1$.
It is well-known since Sylvester's original article \cite{Sylvester1880-OnVulgarFractions} that the sum of the inverses of the first $d$ terms of the Sylvester sequence, \ie $\sum_{i = 1}^d \frac{1}{s_i}$, is the closest possible underestimate of $1$ by any sum of $d$ Egyptian fractions.

\begin{example}\label{ex:Egyptian}
As the properties of the Sylvester sequence indicate, the Sylvester simplex $\Sylv$ is a $k$-Egyptian diagonal $d$-simplex.

Moreover, the simplex $\conv\{\b 0, d\b e_1, \dots, d\b e_d\}$ is a $0$-Egyptian diagonal $d$-simplex for any $d \geq 1$. 

If $\conv(\b 0, a_1\b e_1, \dots, a_{d-1}\b e_{d-1}, a_d\b e_d)$ is $k$-Egyptian, then $\conv(\b 0, a_1\b e_1, \dots, a_{d-1}\b e_{d-1}, \frac{k+2}{k+1}a_d\b e_d)$ is $(k+1)$-Egyptian.
\end{example}

\begin{proposition}\label{prop:SylvesterSimplexBasicFacts}
For $d\geq 1$, $k\geq 0$, any $k$-Egyptian diagonal $d$-simplex $\simplex=\conv(\b 0, a_1\b e_1, \dots, a_d\b e_d)$  
\begin{compactenum}[(a)]

\item is $d$-dimensional and has vertices $\b0$ and $a_i\b e_i$ for $i\in [d]$,\label{item:SylvVertices}\label{item:SylvDimension}

\item is defined by the inequalities
$x_i \geq 0$ for all $i\in[d]$ and $\sum_{i=1}^{d}\frac{1}{a_i}x_i\le 1$,
\label{item:SylvInequalities}

\item has volume $\vol(\simplex) = \tfrac{1}{d!} a_1\cdot a_2\cdots a_d$, and normalized volume $\nVol(\simplex) = a_1\cdot a_2\cdots a_d$,\label{item:SylvVolume}
\item has exactly $k$ interior lattice points, namely $(1, \dots, 1, j)$ for $j \in [k]$, \label{item:SylvInteriorLatticePoints}

\item has a single lattice point in the relative interior of the facet opposing $\b 0$, namely $(1, \dots, 1, k+1)$,\label{item:SylvInteriorLatticePointsInFacets}

\item has a single lattice point in the relative interior of the facet opposing $a_d \b e_d$, namely $(1, \dots, 1, 0)$,\label{item:SylvInteriorLatticePointsInFacets2}

\item is a reflexive polytope if and only if $k=1$ and
$a_i$ divides $a_d$ for all $i\in[d-1]$,\label{item:SylvReflexiveK1}

\item is a smooth polytope if and only if $d=1$ or $\bigl(d \geq 2$, $k = 0$ and $a_i = d$ for all $i\in[d]\bigr)$.\label{item:SylvNeverSmooth}

\end{compactenum}
\end{proposition}
 
\begin{proof}
	We verify the properties one at a time.
\begin{compactenum}[(a)]
\item This holds by definition.

\item Clearly, for all $i\in [d]$, all the vertices of $\simplex$ except $a_i\b e_i$ satisfy the equality $x_i = 0$.
Besides, all vertices of $\simplex$ except $\b 0$ (\ie of the form $a_i\b e_i$) satisfy the equality $\sum_i \frac{1}{a_i}x_i = 1$.

\item Since the normalized volume of $\simplex$ is equal to $\left|\det\left(\begin{pmatrix}1\\\b 0\end{pmatrix}, \begin{pmatrix}1\\ a_1\b e_1\end{pmatrix},\dots, \begin{pmatrix}1\\ a_d \b e_d\end{pmatrix}\right)\right|$ and since the Euclidean volume is equal to the normalized volume divided by $d!$, the claim follows.

\item A lattice point is in the interior if it satisfies all the defining inequalities in  \eqref{item:SylvInequalities} strictly.
It is easy to check that for $j\in [k]$, the point $(1, \dots, 1, j)$ satisfies the conditions.
Conversely, if $(x_1, \dots, x_d)$ is an interior lattice point of $\simplex$, then we must have  $ x_{i} \geq 1 $ for $ i \in [d] $ and  $ \sum_{i} x_{i}/a_{i} < 1 $.
If the first $ d-1 $ entries are equal to one, then the last one can only be in $ [k] $.
It remains to show that none of the first entries can be greater than one.
Suppose $ x_{i} \geq 2 $ with $ i \in [d-1] $.
Then
\begin{equation*}\label{eq:auxilliary_interior_points}
	1 > \sum_{i=1}^d \frac{x_{i}}{a_{i}} \geq \frac{1}{a_{i}} + \sum_{i=1}^d \frac{1}{a_{i}} \geq \frac{(k+1)}{a_{d}} + \sum_{i=1}^d \frac{1}{a_{i}} \geq 1,
\end{equation*}
by the defining properties of $k$-Egyptian simplices \eqref{eq:k_egyptian}.
This contradiction shows that there are only $ k $ interior lattice points.

\item Consider the simplex $\simplex' = \conv(\b 0, a_1\b e_1, \dots, a_{d-1}\b e_{d-1}, \frac{k+2}{k+1}a_d\b e_d)$. It is $(k+1)$-Egyptian (cf. \Cref{ex:Egyptian}) and  $\simplex\subseteq \simplex'$.
In particular, by (d), the simplex $\simplex'$ contains one more interior lattice point than the simplex $\simplex$, namely $(1, \dots, 1, k+1)$.
By the facet description of \eqref{item:SylvInequalities}, this point lies on the facet $\polF$ of $\simplex$ opposite  the vertex $\b 0$.
Since the relative interior of $\polF$ is in the interior of $\simplex'$, this is the unique interior lattice point of $\polF$.
\item Consider the facet $\polF_d=\conv(0,a_1\b e_1,\ldots,a_{d-1}\b e_{d-1})$ opposite to the vertex $a_d\b e_d$. By (b), a lattice point $\b x=(x_1,\ldots,x_{d-1},x_d)$ lies in its relative interior if and only if $x_d=0$, $x_i\geq 1$ for $1\leq i\leq d-1$ and $\sum_{i=1}^{d-1}\frac{1}{a_i}x_i<1$. 
By \eqref{eq:k_egyptian}, we have $\sum_{i=1}^{d-1}\frac{1}{a_i}=1-\frac{k+1}{a_d}<1$, so the point $(1,\ldots,1,0)$ satisfies these conditions. 
If $\b x\neq(1,\ldots,1,0)$ lies in the relative interior of $\polF_d$, then $x_j\geq 2$ for some $1\leq j\leq d-1$. It follows from \eqref{eq:k_egyptian} that 
\[
\sum_{i=1}^{d-1}\frac{x_i}{a_i}\geq \sum_{i=1}^{d-1}\frac1{a_i}+\frac1{a_j}\geq 1-\frac{k+1}{a_d}+\frac{k+1}{a_d}
=1,
\]
which yields a contradiction. In particular, $(1,\ldots,1,0)$ is the unique lattice point in the relative interior of $\polF_d$. 

\item By \eqref{item:SylvInteriorLatticePoints}, if $k \ne 1$, then $\simplex$ does not have a unique interior lattice point, so it is not reflexive.
If $k = 1$, then the polar of $\simplex - (1, \dots, 1, 1)$ is the simplex $\conv\bigl(-\b e_1, \dots, -\b e_d, (\frac{a_d}{a_1}, \dots, \frac{a_d}{a_{d-1}}, 1)\bigr)$: the reader can use~\eqref{item:SylvInequalities}, or refer to \cite[Prop.~4.4]{Nill2007VolumeLatticePointsReflexiveSimplices}.
Hence, for $d\geq1$, $\simplex$ is reflexive if and only if the polar of $\simplex - (1, \dots, 1, 1)$ is a lattice simplex, that is if and only if $a_i$ divides $a_d$ for all $i\in[d-1]$.

\item The case $d = 1$ is trivial.
For $d \geq 2$, suppose that $\simplex$ is smooth. 
Then the primitive vectors of the edges adjacent to $a_i\b e_i$ form a basis of $\Z^d$ for all $i\in [d]$.
For $i = d$, this implies that the determinant of the following matrix is $\pm1$:
$$\begin{pmatrix}
\tfrac{a_1}{\gcd(a_1, a_d)} & & \\
& \ddots & & \\
& & \tfrac{a_{d-1}}{\gcd(a_{d-1}, a_d)} & \\
\tfrac{-a_d}{\gcd(a_1, a_d)} & \dots & \tfrac{-a_d}{\gcd(a_{d-1}, a_d)} & 1
\end{pmatrix}$$
Hence $\prod_{j < d} \tfrac{a_j}{\gcd(a_j, a_d)} = 1$, which yields $\tfrac{a_j}{\gcd(a_j, a_d)} = 1$ for all $j\in [d-1]$. It follows that $a_d$ divides $a_j$ for all $1\leq j <d$. Repeating the determinant argument at every vertex, yields that also $a_d$ divides $a_j$ for all $1\leq j< d$. We conclude that  $a_j = \pm a_d$ for all $j\in [d-1]$.
Hence $a_1 = a_2 = \dots = a_d$ since $0\leq a_1 \leq \dots \leq a_d$.
Since $\simplex$ is $k$-Egyptian, we obtain $k = 0$ by $a_{d-1}\leq \frac{a_d}{k+1}$, and finally, $a_i = d$ for all $i\in [d]$ by $\frac{0+1}{a_d} + \sum_{i=1}^{d-1} \frac{1}{a_i} = 1$.
\qedhere
\end{compactenum}
\end{proof}

\begin{remark}
In the case of the Sylvester simplex, some items of \Cref{prop:SylvesterSimplexBasicFacts} simplify slightly:
\begin{compactenum}
\item[\eqref{item:SylvInequalities}] Since $\frac{s_d - 1}{s_i}$ is an integer for all $i\in [d-1]$, one can write the last facet inequality with integer (co-prime) coefficients as $x_d + \sum_{i=1}^{d-1} \tfrac{(k+1)\cdot (s_d-1)}{s_i}x_i = (k+1)\cdot (s_d - 1)$;

\item[\eqref{item:SylvVolume}] $\vol(\Sylv) = \tfrac{k+1}{d!}\cdot (s_d-1)^2$, and normalized volume $\nVol(\Sylv) = (k+1)\cdot(s_d-1)^2$;

\item[\eqref{item:SylvInteriorLatticePointsInFacets} and \eqref{item:SylvInteriorLatticePointsInFacets2}] Both facets, the one opposing $\b 0$ and the one opposing $s_d \b e_d$ have one lattice point in their relative interior. Namely, $(1,\ldots,1,k+1)$ respectively $(1,\ldots,1,0)$. We will see in \Cref{lem:UnimodularEquivalenceLayering}, that these two facets are even  unimodularly equivalent. %

\item[\eqref{item:SylvReflexiveK1}] Reflexivity holds if and only if $k = 1$;

\item[\eqref{item:SylvNeverSmooth}] Smoothness holds if and only if $d = 1$ or $(d, k) = (2, 0)$.
\end{compactenum}
\end{remark}

\begin{example}
The $0$-Egyptian diagonal $d$-simplex $\simplex = \conv(\b 0, d\b e_1, \dots, d\b e_d)$ is smooth,  has no interior lattice point and hence is not reflexive. It has volume $\vol(\simplex) = \frac{d^d}{d!}$, and is defined by the inequalities $x_i \geq 0$ for all $i\in [d]$ and $\sum_i x_i \leq d$.

Though it has many lattice points in the facet opposing $d\b e_d$, the only one in the relative interior of this facet is $(1,\ldots,1,0)$. The points   $(2, 1, \dots, 1, 0)$, $(1, 2, 1, \dots, 1, 0)$, $(1, 1, 2, 1, \dots, 1, 0)$, $\dots$, $(1, \dots, 1, 2, 0)$ lie on the boundary of this facet since they are also contained in the facet opposing $\b 0$. These are not the only lattice points on the boundary of the facet opposing $\b e_d$.
\end{example}

Among the slightly less obvious properties of the Sylvester simplices are those relative to its widths.
The \defn{width in direction $\b c\in\R^d$} of a polytope $\pol\subseteq\R^d$ is
$$\mathdefn{\omega(\pol, \b c)} \coloneqq \max_{\b x\in \pol}\langle\b x, \b c\rangle - \min_{\b y\in \pol}\langle \b y, \b c\rangle \,.$$

Recall that for $\b x\in \R^d$ and $p\geq 1$ its \defn{$p$-norm} is $\mathdefn{\|\b x\|_p} \coloneqq \left(\sum_{j=1}^d |x_j|^p\right)^{1/p}$. 
The \defn{minimal width} (\aka \defn{thickness}, \defn{minimal breadth}) of $\pol$ is the minimum of $\omega(\pol, \b c)$ where $\b c$ ranges over a certain subset of $\R^d$, \eg the \defn{lattice width} of $\pol$ is $\min\set{\omega(\pol, \b c)}{\b c\in \Z^d \text{ primitive}}$; while its \defn{Euclidean width} is given by $\min\set{\omega(\pol, \b c)}{\b c\in \R^d \text{ with } \|\b c\|_2 = 1}$.
Similarly, the \defn{maximal width} (\aka \defn{diameter}) of $\pol$ is the maximum of $\omega(\pol, \b c)$ where $\b c$ ranges over a certain subset of $\R^d$.
Usually, one guarantees that $\omega(\pol, \b c) \ne 0$ for all $\b c$ in the chosen subset of $\R^d$, and that this subset is compact (implying that minima and maxima are attained).

For $d = 1$, as $\Sylv[k][1] = [0, k+1]$, the unique width of $\Sylv[k][1]$ is $k+1$. 
For $d\geq 2$, we refer to \Cref{tab:Width} for the minimal and maximal widths of $\Sylv$.
We omit the proofs since it is not the main topic of the present article.

\begin{table} 
\setlength{\tabcolsep}{3.5pt}

\begin{adjustbox}{center,max width=0.96\textheight,max totalheight=0.70\textwidth}
\begin{tabular}{@{}c c | c c @{}}
\toprule
Name & $\b c\in$ & minimal width $\omega(\Sylv, \b c)$ & maximal width $\omega(\Sylv, \b c)$ \\
\midrule

Lattice & $\Z^d$, primitive & $2$ & $+\infty$ \\
\addlinespace[2pt]

$1$-norm & $\R^d$, $\|\b c\|_1 = 1$ & $\frac{(k+1)\cdot(s_d-1)}{(k+1)\cdot(s_d-1) - k}$ & $(k+1)\cdot(s_d-1)$ \\
\addlinespace[2pt]

Euclidean & $\R^d$, $\|\b c\|_2 = 1$ & $1/\sqrt{\frac{1}{(k+1)^2\cdot(s_d-1)^2} + \sum_{j < d} \frac{1}{s_j^2}}$ & $\sqrt{s_{d-1}^2 + (k+1)^2\cdot(s_d-1)^2}$ \\
\addlinespace[2pt]

$\infty$-norm & $\R^d$, $\|\b c\|_{\infty} = 1$ & $2$ & $(k+1)\cdot(s_d-1) + s_{d-1}$ \\

\bottomrule
\end{tabular}
\end{adjustbox}
\caption{Minimal and maximal widths of $\Sylv$ for $ d\geq 2$}
\label{tab:Width}
\end{table}

\section{Several ways of dissecting Sylvester simplices}\label{sec:DissectingSylvesterSimplex}

In this section, we study triangulations of Sylvester simplices.

\subsection{A flag, regular, and unimodular triangulation for Sylvester simplices}\label{ssec:ChimneyPolytopesPointOfView}

We begin by showing that $\Sylv$ admits a flag, regular, and unimodular triangulation.
We combine two constructions that preserve having such a triangulation: chimney polytopes and one-point extensions.
We will proceed by induction on $d$.
In what follows, recall that $\Sylv[-1] \subseteq \R^d$ is the $(d-1)$-dimensional simplex whose vertices are $\b0$ and $s_i\b e_i$ for $i\in [d-1]$.

Let $\polytope{Q}\subseteq\R^{d-1}$ be a lattice polytope.
Let $\ell$ and $u$ be two affine linear functionals $\R^{d-1} \to \R$ with integer coefficients such that $\ell \leq u$ on $\polytope{Q}$. 
The associated \defn{chimney polytope} is:
$$\mathdefn{\Chim(\polytope{Q}, \ell, u)} \coloneqq \set{(\b x,y)\in\R^{d-1}\times \R}{\b x \in \polytope{Q}, ~ \ell(\b x)\leq y \leq u(\b x)} ~\subseteq~\R^{d}\, .$$

\begin{example}
Let $\polQ = [0, 2]$ be the line segment.
Define $\ell (x) = 0$ and $u(x) = -x + 3$.
As illustrated in \Cref{fig:ChimneyAndOnePointExtension} (left), we get: $\Chim(\polQ, \ell, u) = \conv\bigl((0,0), (2,0), (2,1), (0,3)\bigr)$.
\end{example}

\begin{figure}
    \centering
    \includegraphics[width=0.4\linewidth]{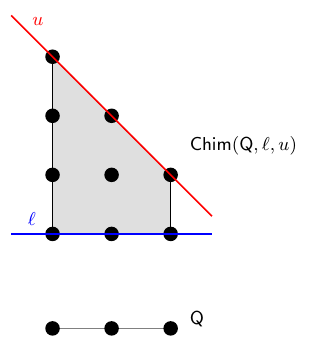}\hspace{1cm}
    \includegraphics[width=0.44\linewidth]{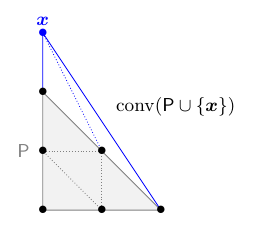}
    \caption{(Left) Example of a chimney polytope of a $1$-dimensional polytope.
    (Right) Example of a one-point extension of the polytope $\pol$ by the point $\b x$: a single facet is visible from $\b x$ (the top-most of $\pol$), and the triangulation of $\pol$ is extended to $\conv(\pol\cup\{\b x\})$ by adding $\conv(\simplex\cup\{\b x\})$ for each simplex $\simplex$ of the triangulation induced on this facet.}
    \label{fig:ChimneyAndOnePointExtension}
\end{figure}

We abuse notation, extending naturally this notion to any polytope $\polQ\subseteq\R^{d}$ contained in the hyperplane $\set{\b x\in \R^{d}}{x_{d} = 0}$.
Namely, if $\polQ \subseteq \set{\b x\in \R^{d}}{x_{d} = 0}$, then we define:
\[
	\Chim(\polytope{Q}, \ell, u) \coloneqq \set{(x_1, \dots, x_{d-1}, y)\in\R^{d}}{(x_1, \dots, x_{d-1}, 0) \in \polytope{Q}, ~ \ell(\b x)\leq y \leq u(\b x)} ~\subseteq~\R^{d}\, .
\]

\begin{proposition}\label{prop:SylvesterSimplexIsChimney}
Fix integers $ d,k $.
Consider the pair of affine linear functionals on $ \R^d $
\begin{equation*}
\ell(\b x) = 0 \qquad \text{and} \qquad u(\b x) = (k+1)\cdot\left(s_d - 1 - \sum_{j = 1}^{d-1} \frac{s_d - 1}{s_j} x_j\right),
\end{equation*}
where $(s_n)_{n\in \N}$ is the Sylvester sequence.
Then we have 
\[
	\Sylv = \Chim(\Sylv[-1], \ell, u)\, .
\]
\end{proposition}

\begin{proof}
First, as $s_d = 1 + \prod_{j < d} s_j$, we indeed get that $\frac{s_d-1}{s_j}$ is an integer for all $j\in[d-1]$, so $\ell$ and $u$ are affine linear functionals with integer coefficients.
We just need to compute the vertices of $\Chim(\Sylv[-1], \ell, u)$ which are of the form $\bigl(\b v, \ell(\b v)\bigr)$ and $\bigl(\b v, u(\b v)\bigr)$ for $\b v$ a vertex of $\Sylv[-1]$.
\begin{compactitem}
\item For $\b v = \b 0$, we have:
$\bigl(\b v, \ell(\b v)\bigr) = \b 0$ and $\bigl(\b v, u(\b v)\bigr) = (k+1)(s_d - 1)\b e_d$.

\item For $\b v = s_i\b e_i$, we have:
$\bigl(\b v, \ell(\b v)\bigr) = s_i\b e_i$ and $\bigl(\b v, u(\b v)\bigr) = s_i\b e_i$.
\end{compactitem}
Hence, $\Sylv$ and $\Chim(\Sylv[-1], \ell, u)$ are simplices with the same vertices: they are equal.
\end{proof}

We now introduce one-point extensions, following the notation we used in \cite[Section~2.4]{CaicedoJuhnkePoullot}.
For a lattice polytope $\pol$ and a lattice point $\b x\notin \pol$, the \defn{one-point extension} of $\pol$ by $\b x$ is the polytope $\conv\bigl(\pol\cup\{\b x\}\bigr)$.
A face $Y\subseteq\pol$ is \defn{visible from $\b x$} if $\conv\bigl(Y\cup\{\b x\}\bigr) \cap \pol = Y$ (see \Cref{fig:ChimneyAndOnePointExtension} (right)).

\begin{lemma}\label{lem:SylvesterSimplexBasisIsOnePointExtension}
The simplex $\Sylv[-1]$ is a one-point extension of $\Sylv[0][d-1] \times \{0\}$ by the point $s_{d-1}\b e_{d-1}$.
The faces of $\Sylv[0][d-1] \times \{0\}$ visible from $s_{d-1}\b e_{d-1}$ are the facet opposite to $\b 0$ in $\Sylv[0][d-1] \times \{0\}$, and all the faces of this facet.
\end{lemma}

\begin{proof}
This is immediate from the definitions:
The $(d-1)$-dimensional simplices $\Sylv[-1]$ and $\Sylv[0][d-1]\times\{0\}$ in $\R^d$ have in common the vertices $\b 0$, and $s_i\b e_i$ for all $i\in[d-2]$; while $\Sylv[-1]$ has a vertex $s_{d-1}\b e_{d-1}$, and $\Sylv[0][d-1]$ has a vertex $(s_{d-1}-1)\b e_{d-1}$.
\end{proof}

To prove \Cref{thmA} we make use of the following two results

\begin{compactenum}[(a)]
\item {\cite[Thm.~2.8 \& Cor.~2.9]{Unimodular_triangulation_existence}} If $\polytope{Q}$ admits a flag, regular, and unimodular triangulation, then $\Chim(\polytope{Q}, \ell, u)$ also admits a flag, regular, and unimodular triangulation, for all affine linear functionals $\ell$ and $u$ with integer coefficients such that $\ell\leq u$ on $\polytope{Q}$.
\label{item:fruTriangulationViaChimney}

\item \cite[Thm.~2.7]{CaicedoJuhnkePoullot} If $\pol$ admits a flag and regular triangulation $\c T$, then the one-point extension by the point $\b x$, \ie $\conv(\pol\cup\{\b x\})$, also admits one.
Moreover, the simplices of this triangulation of $\conv(\pol\cup\{\b x\})$ are the simplices of $\c T$ together with all the simplices of the form $\conv(\simplex \cup\{\b x\})$ for $\simplex\in\c T$ a simplex in a face of $\pol$ visible from $\b x$.\label{item:frTriangulationOnePointExtension}
\end{compactenum}

First, we focus on the Sylvester bases $\Sylv[-1]$.

\begin{proposition}\label{prop:InductionStepViaOnePointExtension}
If $\Sylv[0][d-1]$ admits a regular, flag, and unimodular triangulation, then $\Sylv[-1]$ admits a regular, flag, and unimodular triangulation.
\end{proposition}

\begin{proof}
Combining \eqref{item:frTriangulationOnePointExtension} \cite[Thm.~2.7]{CaicedoJuhnkePoullot} and \Cref{lem:SylvesterSimplexBasisIsOnePointExtension} guarantees that a regular and flag triangulation $\c T_\circ$ of $\Sylv[0][d-1]$ can be extended into a regular and flag triangulation $\c T$ of $\Sylv[-1]$.
It remains to prove that $\c T$ stays unimodular.

The $(d-2)$-dimensional simplex $\polF \coloneqq \conv\bigl(s_1\b e_1, \dots s_{d-2}\b e_{d-2}, (s_{d-1}-1)\b e_{d-1}\bigr)$ and the $(d-1)$-dimensional simplex $\conv\bigl(s_1\b e_1, \dots, s_{d-2}\b e_{d-2}, (s_{d-1}-1)\b e_{d-1}, s_{d-1}\b e_{d-1}\bigr)$ have the same volume (in their respective lattices), namely $\prod_{i\leq d-2} s_i$. 
The same comparison of volumes via determinants show that if $\simplex$ is a simplex of $\c T_\circ$ lying inside the facet $\polF$ of $\Sylv[0][d-1]$, then $\conv(\simplex, s_{d-1}\b e_{d-1})$ and $\simplex$ have the same volume (in their respective lattices).
Thus, by \eqref{item:frTriangulationOnePointExtension} \cite[Thm.~2.7]{CaicedoJuhnkePoullot} and \Cref{lem:SylvesterSimplexBasisIsOnePointExtension}, if $\c T_\circ$ is unimodular, then so is $\c T$ because the only faces of $\Sylv[0][d-1]$ visible from $s_{d-1}\b e_{d-1}$ lie in $\polF$.
This proves the claim and thus the proposition.
\end{proof}

We generate the Sylvester simplices iteratively by taking one-point extensions and chimney polytopes while maintaining a flag, regular and unimodular triangulation, as illustrated in \Cref{fig:SylvSimplicesSequence}.
This allows us to prove \Cref{thmA}.

\begin{figure}
    \centering
    \includegraphics[width=0.99\linewidth]{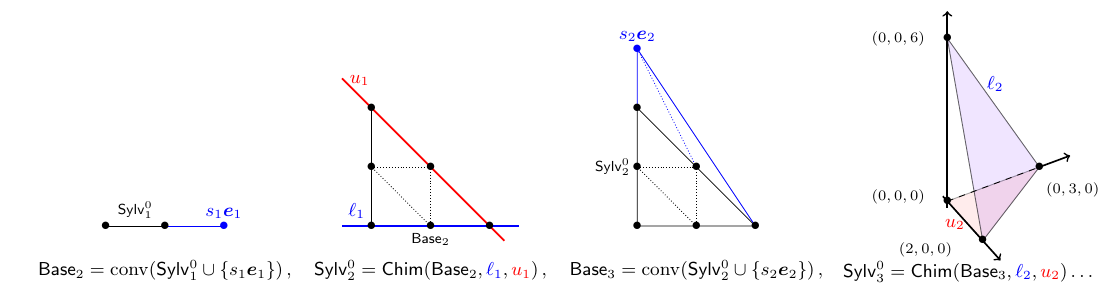}
    \caption{The sequence $(\Sylv[0][1], \Sylv[-1][1], \Sylv[0][2], \Sylv[-1][2], \Sylv[0][3],\dots)$ can be constructed taking iteratively one-point extensions and chimney polytopes.
    \vspace{-0.4cm}
    }
    \label{fig:SylvSimplicesSequence}
\end{figure}

\begin{theorem}\label{thm:fruTriangulationSylvesterSimplices}
The simplex $\Sylv$ admits a flag, regular, and unimodular triangulation for all $d\geq 1$ and $k\geq 0$.
The simplex $\Sylv[-1]$ also admits a flag, regular, and unimodular triangulation for $d\geq 1$.
\end{theorem}

\begin{proof}
As it is a segment, the simplex $\Sylv[k][1] = [0, k+1]$ admits a flag, regular, and unimodular triangulation for all $k\geq 0$, and so (trivially) does $\Sylv[-1][1] = \{0\}$.

Suppose that, for a fixed $d$, and for all $k$, all simplices $\Sylv$ admit a flag, regular, and unimodular triangulation.
By \Cref{prop:InductionStepViaOnePointExtension}, as in particular $\Sylv[0][d]$ admits one, so does  $\Sylv[-1][d+1]$. 
Combining \eqref{item:fruTriangulationViaChimney} \cite[Thm.~2.8 \& Cor.~2.9]{Unimodular_triangulation_existence} and \Cref{prop:SylvesterSimplexIsChimney}, we deduce that $\Sylv[k][d+1]$ admits a flag, regular and unimodular triangulation for all $k\geq 0$.
The claim follows by induction.
\end{proof}

\begin{corollary}
For $d\geq 1$, $k\geq 0$, both $\Sylv$ and $\Sylv[-1]$ have the integer decomposition property.
\end{corollary}

\begin{proof}
Having a regular and unimodular triangulation implies having the integer decomposition property, see \cite[Figure~1]{FerroniHigashitani2024ExamplesCounterexamplesEhrhartTheory} for a comprehensive survey.
\Cref{thm:fruTriangulationSylvesterSimplices} yields the claim.
\end{proof}

\subsection{The layering of Sylvester simplices}\label{ssec:UnimodularLayering}

In \Cref{ssec:ChimneyPolytopesPointOfView}, we described the Sylvester simplex as a chimney polytope $\Sylv = \Chim(\Sylv[-1], \ell, u)$ for affine linear functionals $\ell$ and $u$ with $\ell\leq u$ (on $\Sylv$). 
Yet, we did not exploit the fact that $\ell(\b v) = u(\b v)$ for all but one of the vertices $\b v$ of $\Sylv[-1]$ (see \Cref{fig:SylvSimplicesSequence}).
This allows for a finer description of how $\Sylv$ depends on $k$.
Therefore, we now construct a triangulation of $\Sylv$ using copies of $\Sylv[0]$: although this triangulation is not unimodular, it will be much  more convenient in the discussion thereafter.
See \Cref{fig:Layering} for an illustration.
We denote $\mathdefn{\overline{\pol[X]}}$ the closure of $\pol[X]\subseteq\R^d$.

\begin{figure}
    \centering
    \includegraphics[width=0.99\linewidth]{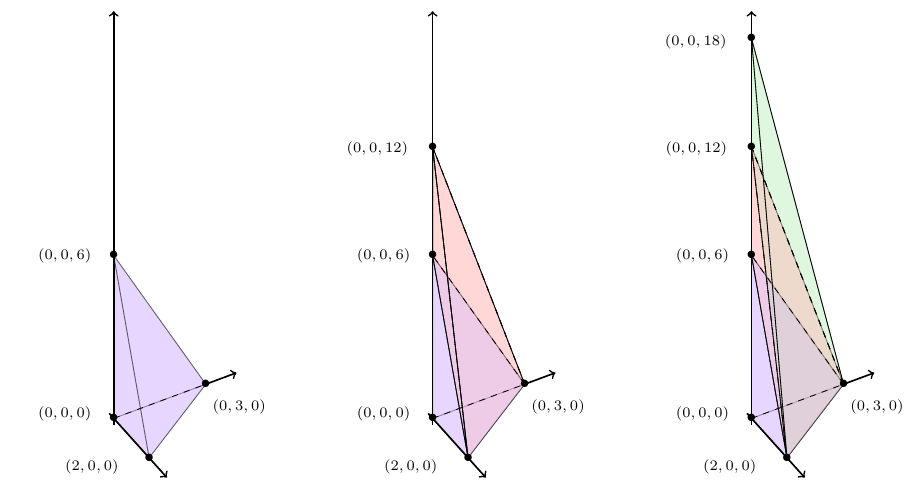}
    \caption{The simplices $\Sylv[0][3]$, $\Sylv[1][3]$ and $\Sylv[2][3]$ drawn via the layering of \Cref{lem:UnimodularEquivalenceLayering} (the scale on the vertical axis is not linear, but only indicative).}
    \label{fig:Layering}
\end{figure}

\begin{lemma}\label{lem:UnimodularEquivalenceLayering}
Let $d\geq2$ and $k\geq1$, and let $\polF_d^{k-1}$ denote the facet of
$\Sylv[k-1][d]$ opposite to $\b 0$. There is an affine
unimodular equivalence
$ \Sylv[0]
 \simeq_{\mathbb Z}
 \overline{\Sylv[k][d]
            \ssm\Sylv[k-1][d]}$
which maps $\Base$ onto $\polF_d^{k-1}$. Consequently,
\[
 \Sylv[k][d]
 \ssm\Sylv[k-1][d]
 \simeq_{\mathbb Z}
 \Sylv[0]\ssm\Base.
\]
In particular, $\Base$ is unimodularly equivalent to
the facet $\polF_d^0$ of $\Sylv[0]$ opposite to $0$.
\end{lemma}

\begin{proof}
We define $\b v'_i = s_i\b e_i$ for $i\in [d-1]$, and $\b v'_0 = k\cdot(s_d - 1)\b e_d$ and $\b v'_d = (k+1)\cdot (s_d - 1)\b e_d$.
The set $\Sylv\ssm \Sylv[k-1]$ is the half-open simplex on these vertices, \ie it is the simplex $\conv(\b v'_i \;|\; 0\leq i\leq d)$ without the facet opposing the vertex $\b v'_d$.

Similarly, we define $\b v_i = s_i\b e_i$ for $i\in [d-1]$, and $\b v_0 = \b 0$ and $\b v_d = (s_d-1) \b e_d$, which are the vertices of the half-open simplex $\Sylv[0]\ssm \Sylv[-1]$.
Consider the $(d+1)\times(d+1)$ matrix:
$$U = \begin{pmatrix}
1 & 0 & \dots & \dots & 0 \\
0 & \ddots & \ddots && \vdots \\
\vdots & \ddots & \ddots & \ddots & \vdots \\
0 & \dots & 0 & 1 & 0\\
+k\tfrac{s_d - 1}{s_0} & -k\tfrac{s_d - 1}{s_1} & \dots & -k\tfrac{s_d - 1}{s_{d-1}} & 1
\end{pmatrix}$$

It is easily seen that $U\begin{pmatrix} 1 \\ \b v_i\end{pmatrix} = \begin{pmatrix} 1 \\ \b v'_i\end{pmatrix}$, since $s_0 = 1$.
Moreover, as $s_d - 1 = \prod_{i < d} s_i$, all entries of $U$ are integers. We also have $\det U = 1$. It follows that the affine unimodular transformation $U$ maps $\Sylv[0][d]$ to $\overline{\Sylv[k][d] \ssm\Sylv[k-1][d]}$, where $\Base$ is mapped onto $\polF_d^{k-1}$. The claim follows.
\end{proof}

We can now prove \Cref{thmB}.

\begin{theorem}\label{thm:HaveAffineCoefficients}
Let $d \geq 2$.
Each coefficient of the polynomials $\Ehr(\Sylv;t)$,  $h^\ast(\Sylv; t)$ and $f^\ast(\Sylv; t)$ as well as the number of lattice points $\bigl|\Sylv\cap \Z^d\bigr|$ is an affine function in $k$. In the last three cases, the coefficients of these functions are non-negative integers.
\end{theorem}

\begin{proof}
This statement may seem obvious to an expert from the decomposition proven in \Cref{lem:UnimodularEquivalenceLayering}, but we detail it here for traceability.
We can compute $h^\ast(\Sylv; t)$  via the following decomposition into half-open $d$-simplices
$$\Sylv = \Sylv[0] \sqcup \bigsqcup_{p = 1}^{k} (\Sylv[p]\ssm \Sylv[p-1]).$$
Let $H_1(t) = h^\ast(\Sylv[0]; t)$ and $H_2(t) = h^\ast(\Sylv[0] \ssm\Sylv[-1]; t)$.
Using the unimodular equivalence of \Cref{lem:UnimodularEquivalenceLayering}, we get 
\[
h^\ast(\Sylv; t) = h^\ast(\Sylv[0];t) + \sum_{p = 1}^{k} h^\ast(\Sylv[p]\ssm \Sylv[p-1];t) = H_1(t) + \sum_{p=1}^k H_2(t) = H_1(t) + k\cdot H_2(t),
\]
where the $h^\ast$-polynomial of a half-open simplex is defined analogously to the usual one (see \eg \cite[Chapter 5.5]{BeckSanyal}). As $H_1(t)$ and $H_2(t)$ have non-negative integer coefficients (see \cite[Equation (5.5.1)]{BeckSanyal}), $h_i^\ast(\Sylv)$ is an affine function in $k$ with non-negative integral coefficients. 

As the coefficients of $f^\ast(\Sylv; t)$ are non-negative integral linear combinations of the coefficients of $h^\ast(\Sylv; t)$, they are also affine functions  in $k$ with non-negative integral coefficients.
Similarly, as the coefficients of the Ehrhart polynomial can be written as linear combinations of the $h^\ast$-vector entries, it follows that also these coefficients are affine functions in $k$. 
Finally, since $|\Sylv\cap\Z^d| = h^\ast_1(\Sylv) + d + 1$, the number of lattice points of $\Sylv$ is also an affine function in $k$ with non-negative integral coefficients.
\end{proof}

We now fix some notation that will be used in the following sections. 
As, by \Cref{thm:HaveAffineCoefficients}, $h_i^\ast(\Sylv)$ and $f_i^\ast(\Sylv)$ are affine functions in $k$ with non-negative integral coefficients, there exist non-negative integers $a_{d, i}$, $b_{d, i}$, $x_{d, i}$ and $y_{d, i}$ such that

\begin{equation*}\label{eq:affine}
h_i^\ast(\Sylv)=a_{d, i}\cdot k + b_{d, i} \quad \text{ and }\quad f_{i}^\ast(\Sylv)=x_{d, i}\cdot k+y_{d, i}  \,.
\end{equation*}

We will  prove in \Cref{cor:Symmetry_h_star} and \Cref{thm:f_star_induction} that 
\[
a_{d, i} = b_{d, d-i} \text{ for all } 0\leq i\leq d \text{ and }  y_{d, i} - x_{d, i} = 1 + \sum_{p=0}^{i} y_{d-i-1+p,\,\, p} \text{ for all } 0 \leq i < d. 
\]

We invite the reader to check these statements for $d\leq 6$ in \Cref{tab:hstar_coefficients} and \Cref{tab:f_star_coefficients}.
In the upcoming sections, we prove other facts that the reader might discover or verify by looking at the data.

\begin{remark}\label{rmk:UnimodLayeringForDiagonalSimplices}
Note that \Cref{thm:HaveAffineCoefficients} also implies that $\vol(\Sylv)$ is an affine function (actually it is linear in $k+1$ since $\vol(\Sylv[-1]) = 0$ because $\Sylv[-1]$ is not $d$-dimensional).
This is straightforward because we have seen that: $\nVol(\Sylv) = (k+1)\cdot (s_d - 1)^2$.
\end{remark}

\begin{remark}
Fix a family of diagonal $d$-simplices $\simplex_k = \conv(\b 0, a_1\b e_1, \dots, (k+1)a_d\b e_d)$ such that $a_i$ divides $a_d$ for all $i\in [d-1]$.
The proof of \Cref{lem:UnimodularEquivalenceLayering} shows that the simplex $\simplex_k$ admits the same type of layering as the Sylvester simplex.
Precisely: $\simplex_{k+1}\ssm \simplex_k ~\UniEq~ \simplex_{k}\ssm\simplex_{k-1}$ for all $k \geq 1$, and the Ehrhart-theoretic quantities (number of lattice points, coefficients of the Ehrhart, $h^\ast$- and $f^\ast$-polynomials, etc.) are affine functions of $k$.

Note that this does not require $\simplex_k$ to be $k$-Egyptian.
\end{remark}

\begin{remark}\label{ssec:example_small_dimension}
Using \Cref{thm:HaveAffineCoefficients}, we can compute the number of lattice points of $\Sylv$, and the polynomials $\Ehr(\Sylv; t)$, $h^\ast(\Sylv; t)$ and $f^\ast(\Sylv; t)$ for all $k$ by evaluating them only at $k = 0$ and $k = 1$ (via a computer experiment).
We refer to \Cref{tab:LatticePointsNbr,tab:ehrhart_coefficients,tab:magic_coefficients,tab:local_h_coefficients,tab:boundary_h_coefficients,tab:hstar_coefficients,tab:f_star_coefficients} of Appendix~\ref{app:Tables} for the explicit values.

For $d \leq 4$, we can directly enumerate all the lattice points and compute the desired quantities from there.
However, for $d \in\{5, 6, 7\}$, there are too many lattice points for such a straightforward method to work:
instead, we use \Cref{thm:hStarSylvester} and its consequences.

The leading coefficient of $\Ehr(\Sylv; t)$ and $f_d^\ast(\Sylv)$ can be directly deduced from $\nVol(\Sylv)$:
$$d!\cdot [t^d] \Ehr(\Sylv; t) =  f_d^\ast(\Sylv) = \nVol(\Sylv) = (k+1)\cdot (s_d-1)^2\,.$$
Similarly, the number of lattice points of $\Sylv$ can be read  directly from $\Ehr(\Sylv; t)$, $h^\ast(\Sylv; t)$ and $f^\ast(\Sylv; t)$ via
$$\Ehr(\Sylv; 1) = d+1 + h_1^\ast(\Sylv) = f_0^\ast(\Sylv)= \bigl|\Sylv\cap \Z^d\bigr|. $$
\end{remark}

\section{The \texorpdfstring{$f^\ast$}{f*}-vectors of Sylvester simplices }\label{sec:fStarPolynomialsSylvesterSimplices}

In this section, we show a partial recursive formula for the $ f^\ast $-vectors of Sylvester simplices.
To do so, we will compute $f_i^\ast(\Sylv)$ via a unimodular triangulation of $\Sylv$ (see \Cref{subs:fandh}, especially \Cref{rmk:BekteMcMullen}).
Some immediate properties can be derived for $f^\ast(\Sylv)$, namely:

\begin{proposition}\label{prop:BasicPropertiesFstar}
If $d\geq 1$, $k\geq 0$, then $f_{d}^\ast(\Sylv)=(k+1)\cdot(s_d - 1)^2$ and  $\sum_{i=-1}^{d} (-1)^i f_i^\ast(\Sylv) = 0$.
Consequently, 
\[
x_{d, d} = y_{d, d} = (s_d - 1)^2 \quad \text{and} \quad \sum_{i=-1}^{d} (-1)^i \cdot x_{d, i} = \sum_{i = -1}^{d} (-1)^i \cdot y_{d, i} = 0 \,.
\]
\end{proposition}

We invite the reader to verify these statements for $d\leq 6$ in \Cref{tab:f_star_coefficients}. 

\begin{proof}
Since, by \Cref{thm:fruTriangulationSylvesterSimplices},  $\Sylv$ admits a unimodular triangulation, $f_{d}^\ast(\Sylv)$ is equal to the number of maximal simplices of any unimodular triangulation of $\Sylv$, which is equal to its (normalized) volume, \ie $(k+1)\cdot(s_d - 1)^2$ (cf. \Cref{subs:fandh}).
The second statement follows since $\sum_{i=-1}^{d} (-1)^if_i^\ast(\Sylv)$ is equal to the reduced Euler characteristic of the unimodular triangulation, which is $0$. Comparing the constant and linear terms in $k$ proves the last two assertions.
\end{proof}

\begin{remark}\label{lem:fstar_Base}
We want to point out that the definition of $\Sylv$ makes sense even for $k=-1$. By definition, we have 
$\pol[Sylv]_d^{-1}=\Base$, which in turn implies that $f_i^\ast(\Base)=f_i^\ast(\pol[Sylv]_d^{-1})$ for $i\geq 0$.
\end{remark}

We now state a recursion for the $f^\ast$-vector entries.
\begin{theorem}\label{thm:f_star_induction}
For all integers $d \geq 2$ and $0\leq i\leq d$ we have the following recursion:
\[
f^\ast_i(\Sylv[-1]) = f^\ast_{i}(\Sylv[0][d-1]) + f^\ast_{i-1}(\Sylv[-1][d-1]) = \sum_{q=-1}^{i}f_{q}^\ast(\Sylv[0][d-1-i+q]).
\]
Consequently, for all $d\geq 2$ and $0\leq i\leq d-1$, we have
\[
y_{d,i}-x_{d,i} ~=~ y_{d-1,i}+y_{d-1,i-1}-x_{d-1,i-1} ~=~ \sum_{q=-1}^{i} y_{d-1-i+q,\,q}
\]
where we set $y_{d,-1}=1$ and $x_{d,-1}=0$.
\end{theorem}

\begin{proof}
We show the claim by induction on $d$. For $d = 2$, one verifies the formula on \Cref{tab:f_star_coefficients}.

Let $d \geq 3$. Consider a unimodular triangulation $\c T$ of $\Sylv[0][d-1]$. Since, by \Cref{lem:SylvesterSimplexBasisIsOnePointExtension}, $\Sylv[-1]$ is a one-point extension of $\Sylv[0][d-1]$ where only the facet $\polF_{d-1}$ opposite to $\b 0$ is visible, we can extend $\c T$ to a triangulation of $\Sylv[-1]$ by just coning over the simplices in $\polF_{d-1}$ with the vertex $s_{d-1}\b e_{d-1}$. Since, by \Cref{lem:UnimodularEquivalenceLayering},  the facets $\polF_{d-1}$ and $\Sylv[-1][d-1]$ are unimodular equivalent, we can interpret the triangulation on $\polF_{d-1}$ that is induced by $\c T$ as a triangulation of $\Sylv[-1][d-1]$. This shows  the first claimed formula.  Applying the induction hypothesis yields the second claim. The last claim follows by substituting $f_i^\ast(\Sylv) = x_{d,i}\cdot k + y_{d,i}$. 
\end{proof}

Recall that, by \Cref{thm:HaveAffineCoefficients}, the number of lattice points $\bigl|\Sylv\cap \Z^d\bigr|$ is an affine function of~$k$.
We prove the recursive formula on its coefficients that has been hinted at by \Cref{tab:LatticePointsNbr}.

\begin{corollary}\label{cor:NbrLatticePoints_RecursionFormula}
Let $\bigl|\Sylv\cap \Z^d\bigr| = u_d\cdot k + v_d$ be the number of lattice points of $\Sylv$ and set $v_{-1}\coloneqq 0$. Then, for all $d \geq 2$,  we have:
$$v_d - u_d = v_{d-1} + 1.$$
\end{corollary}

\begin{proof}
This follows from \Cref{thm:f_star_induction} combined with the fact that $\bigl|\Sylv\cap \Z^d\bigr| = f^\ast_{0}(\Sylv)$.
\end{proof}

We already know the explicit values for $f_{-1}^\ast(\Sylv)$ (which is $1$) and $f_{d}^\ast(\Sylv)$ (see \Cref{prop:BasicPropertiesFstar}). 
  \Cref{thm:f_star_induction} allows us to determine a third entry of $f^\ast(\Sylv)$ explicitly  for all $d, k\geq 0$.
The result will use the following easy lemma.

\begin{lemma}\label{lem:helper}
Let $\c T$ be a triangulation of a $d$-polytope. Let $f_{d-1}^{\partial}$ (resp. $f_{d-1}^\circ$) be the number of $(d-1)$-simplices of $\c T$ lying in exactly one (resp. two) $d$-simplices of $\c T$.
Moreover, let $f_d$  be the total number of $d$-simplices of $\c T$. Then:
\[
f_{d-1}^\circ+f_{d-1}^\partial=\frac{1}{2}\left((d+1)\cdot f_{d}+f_{d-1}^\partial\right).
\]
\end{lemma}

\begin{proof}
We count pairs $(F, G)$, where $G\subseteq F\in \c T$ and $\dim F=\dim G+1=d$ in two ways. On the one hand, as each $d$-simplex $F\in\c T$ has $d+1$ faces of dimension $d-1$, there are $(d+1)f_d$ such pairs. On the other hand, by definition, this number is equal to $2f_{d-1}^\circ+f_{d-1}^\partial$.
\end{proof}

\begin{proposition}\label{prop:fStarExplicitValues}
For all $d\geq 1$, we have
\[
x_{d,d-1}
  =\frac{d+2}{2}(s_{d+1}-1)-\frac{d+3}{2}(s_d-1) \quad \text{and}\quad
  y_{d,d-1}
=\frac{d+2}{2}(s_{d+1}-1)-\frac{d+1}{2}(s_d-1).
\]
Moreover, for all $d\ge2$,
\[
  y_{d,d-2}-x_{d,d-2}=\frac{d+1}{2}(s_d-1).
\]
\end{proposition}

\begin{proof}
By \Cref{lem:fstar_Base}, the difference $y_{d,d-2}-x_{d,d-2}$ is the number of $(d-2)$-dimensional faces of a unimodular triangulation of the $(d-1)$-simplex $\Sylv[-1]$. Applying \Cref{lem:helper} and using that the triangulation is unimodular we get
\[ 
y_{d,d-2}-x_{d,d-2}=
    f^\ast_{d-2}(\Base)=\frac{1}{2}\left(d\cdot f^\ast_{d-1}(\Base)+f_{d-2}^\partial(\Base)\right).
\]
We note that $f^\ast_{d-1}(\Base)=\nVol(\Base)=\prod_{i < d} s_i = s_d-1$ and since the restriction of $\c T$ to the boundary of $\Base$ is also unimodular, that $f_{d-2}^\partial(\Base)=\nVol(\partial\,\Base)$. 
The latter normalized volume can be computed as follows: Facets of $\Base$ the form $\conv\bigl(\{\b 0\}\cup\set{s_i\b e_i}{i\in [d-1]\ssm\{j\}}\bigr)$  for some $j\in [d-1]$ have normalized volume $\prod_{i\ne j} s_i = \frac{s_d-1}{s_j}$. It remains to compute the normalized volume of the facet $\conv(s_i\b e_i~|~i\in [d-1])$ opposing $\b 0$ in $\Sylv[-1]$.
Since the numbers $s_i$ are pairwise co-prime, a direct computation yields that it has normalized volume $1$. 
Consequently, we have:
$$\nVol(\partial\,\Sylv[-1]) = 1 + \sum_{j=1}^{d-1}\frac{s_d-1}{s_j} = 1 + (s_d-1)\cdot\left(1 - \frac{1}{s_d-1}\right) = s_d - 1 \,$$
and we obtain the claimed formula for $y_{d,d-2}-x_{d,d-2}$.
Applying \Cref{thm:f_star_induction} in dimension $d+1$ and
at index $i=d-1$ and using the  already proven expression, we obtain
\[
  \frac{d+2}{2}(s_{d+1}-1)
  =y_{d,d-1}+\frac{d+1}{2}(s_d-1),
\]
which is the claimed formula for $y_{d,d-1}$.  Finally,
\[
  y_{d,d-1}-x_{d,d-1}=f_{d-1}^\ast(\Base)=\nVol(\Base)=s_d-1,
\]
and the formula for $x_{d,d-1}$ follows.
\end{proof}

\begin{remark}
For a fixed $d$, there are $2\cdot (d+1)$ coefficients $x_{d, i}, y_{d, i}$ to compute to get the complete characterization of $f^\ast(\Sylv; t)$.
Six of these coefficients have an explicit value given in \Cref{prop:fStarExplicitValues}.
If we assume that we know all the coefficients for smaller dimensions, then we only need to compute $d-2$ new coefficients to get these $2\cdot(d+1)$ coefficients:
using the induction of \Cref{thm:f_star_induction} and the alternating sum of the coefficients, we would get the remaining ones.
\end{remark}

\section{The \texorpdfstring{$h^\ast$}{h*}-polynomials of Sylvester simplices are unimodal}\label{sec:Computing_hStar}

The goal of this section is to prove an explicit formula for the $h^\ast$-vector of a Sylvester simplex, show that it satisfies a certain symmetry, which will be proven via a more general statement about half-open Ehrhart-MacDonald reciprocity and finally, to prove that the $h^\ast$-vector is unimodal.
We start with a simple fact on the degree of the $h^\ast$-polynomials of $k$-Egyptian diagonal $d$-simplices.

\begin{proposition}\label{cor:Degree_hStarSylveterSimplices}
If $\polS$ is a $k$-Egyptian diagonal $d$-simplex, then $h^\ast(\polS; t)$ has degree $d$ if $k \geq 1$ and degree  $d-1$ if $k = 0$.
\end{proposition}

\begin{proof}
 It is well-known (see \eg \cite[Theorem~4.5]{Beck_book}) that the degree of  $h^\ast(\polS; t)$ is  $\dim \polS - \ell + 1$, where $\ell\geq1$ is the smallest integer such that $\ell\cdot \polS$ contains a lattice point in its interior.
Since, by \Cref{prop:SylvesterSimplexBasicFacts}~\eqref{item:SylvInteriorLatticePoints}, the simplex $\polS$ has at least one interior lattice point  for $k\geq 1$, the claim follows in this case.

For $k = 0$, the simplex $\polS$ has no interior lattice point, but the simplex $2\cdot\polS$ has one because it contains the facet of $\polS$ opposite to $\b 0$ in its interior which itself contains a lattice point in its relative interior by \Cref{prop:SylvesterSimplexBasicFacts}~\eqref{item:SylvInteriorLatticePointsInFacets}.
Hence, the polynomial $h^\ast(\polS; t)$ has degree $\dim\polS - 1 = d-1$.
\end{proof}

We now show that $h^\ast$-polynomials of diagonal $d$-simplices have a special form.
Essentially this follows from the fact that the corresponding half-open parallelepiped are particularly well-behaved.

\begin{proposition}\label{prop:hStarDiagonalSimplices}
For integers $a_1,\dots, a_d > 0$, let $\psi(\b x) = \left\lceil\sum_{j=1}^d \frac{x_j}{a_j}\right\rceil$ for $\b x\in \Z^d$.
The $h^\ast$-polynomial of the diagonal simplex $\polD = \conv(\b 0, a_1 \b e_1, \ldots, a_d \b e_d)$ is given by:
$$h^\ast(\polD; t) = \sum_{\substack{\b x\in\Z^d\\0\leq x_j < a_j}} t^{\psi(\b x)}\, .$$
\end{proposition}

\begin{proof}
Using \cite[Corollary~3.11]{Beck_book}, we can compute the $h^\ast$-polynomial as follows
$$h^\ast(\polytope{D};t) = \sum_{(\b x, m)\in \Pi^\circ_{\polytope{D}}\cap\Z^{d+1}} t^m = \sum_{\substack{\b x\in\Z^d\\0\leq x_j < a_j}} \left(\sum_{\substack{m\in\N\\(\b x, m)\in \Pi^\circ_{\polytope{D}}\cap\Z^{d+1}}} t^m\right)\, .$$

We now evaluate the innermost sum.
Let $(\b x, m)\in\Pi^\circ_{\polytope{D}}\cap\Z^{d+1}$.
For $1\leq j\leq d$, since $x_j = a_j\lambda_j$ and $0\leq \lambda_j<1$, each $x_j$ is an integer satisfying $0\leq x_j<a_j$.
Hence, $\lambda_j = \frac{x_j}{a_j}$ and thus, $m = \lambda_0 + \sum_{j=1}^d \frac{x_j}{a_j}$.
Since $0\leq\lambda_0<1$ and $m\in\Z$, it follows that $m$ is uniquely determined by $\b x$ with $m = \psi(\b x)$, and the proof is complete.
\end{proof}

This general formula for computing the $h^\ast$-polynomial of diagonal simplices is the one we implemented to compute the $h^\ast$-vector of $\Sylv$  for $d \in\{5,6\}$, see \Cref{tab:hstar_coefficients}.
From there we derived the other Ehrhart-theoretic quantities for these dimensions.
Even though the following formulas (in \Cref{thm:hStarSylvester} and later in \Cref{thm:NumberOfLatticePoints,prop:local_h_star_symmetry,prop:boundaryhStarSylvester}) seem very unappealing as one needs to sum over $\prod_{i< d} s_i = s_d-1$ points,  this is a tremendous reduction from the initial complexity: a naive computation of  $h^\ast(\Sylv; t)$ via the half-open paralellepiped method would run through all points $\b x\in \Z^d$ with $1\leq x_j < s_j$ for $j\in [d]$, which is $\prod_{j \leq d} s_j = s_d(s_d-1) \approx s_d^2$ instead of our new complexity of approximately $s_d$.
This makes dimensions $5$ and $6$ reachable.

\begin{definition}
For a vector $\b x\in \Z^{d-1}$, we set
\[
	\mathdefn{N(\b x)} \coloneqq \sum_{j=1}^{d-1} x_j \frac{s_d-1}{s_j} ~~\in \Z\,,
\]
and we write its Euclidean division by $s_d-1$ as:
\[
	N(\b x) = \mathdefn{q(\b x)}\cdot (s_d - 1) + \mathdefn{r(\b x)}, ~~\text{where}~~ q(\b x), r(\b x)\in \N \text{, with } 0\leq r(\b x) < s_d - 1 \,.
\]
\end{definition}

\begin{theorem}\label{thm:hStarSylvester}
For all $k\geq 0$:
$$h^\ast(\Sylv; t) = \sum_{\substack{\b x\in \Z^{d-1} \\ 0\leq x_j < s_j}} \varphi_k(\b x; t) ~,$$
where
$$ 
\mathdefn{\varphi_k(\b x; t)} \coloneqq
\begin{cases}
t^{q(\b x)}+\bigl((k+1)\cdot(s_d-1)-1\bigr)t^{q(\b x)+1}, & \text{if } r(\b x)=0, \\
\bigl((k+1)(s_d-1 - r(\b x))+1\bigr)t^{q(\b x)+1} + \bigl((k+1)r(\b x)-1\bigr)t^{q(\b x)+2}, & \text{if } 1\leq r(\b x) < s_d-1.
\end{cases}
$$
\end{theorem}

\begin{proof}
For the sake of readability, we write $M \coloneqq s_d-1$ in this proof.
By \Cref{prop:hStarDiagonalSimplices}, we get:
\begin{equation}\label{eq:Proof_hStarSylvesterSimplexDecompositionFormula}
h^\ast(\Sylv; t) = \sum_{\substack{\b x\in \Z^{d-1} \\ 0\leq x_j < s_j}} \sum_{x_d = 0}^{(k+1)M - 1} t^{\left\lceil \sum_{i=1}^{d-1}\frac{x_i}{s_i}+\frac{x_d}{(k+1)M}\right\rceil}.
\end{equation}

Fix $\b x\in \Z^{d-1}$ with $0\leq x_j < s_j$ for all $j\in [d-1]$.
For each $x_d$ the associated exponent is: 
\[
	\left\lceil \frac{1}{M}N(\b x) + \frac{x_d}{(k+1)M}\right\rceil = q(\b x)+\left\lceil \frac{(k+1)r(\b x) + x_d}{(k+1)M}\right\rceil.
\]
We distinguish two cases. \\
{\sf Case 1: $r(\b x) = 0$.} In this situation, we have $\left\lceil \frac{(k+1)r(\b x)+x_d}{(k+1)M}\right\rceil = \left\lceil \frac{x_d}{(k+1)M}\right\rceil$.
Therefore, if $x_d = 0$, the associated exponent equals $q(\b x)$, while for $1 \leq x_d \leq (k+1)M - 1$, the associated exponent equals $q(\b x) + 1$.
Thus, in this case, the innermost sum in \eqref{eq:Proof_hStarSylvesterSimplexDecompositionFormula} simplifies to:
$$t^{q(\b x)}+\bigl((k+1)M-1\bigr)t^{q(\b x)+1}\, .$$

\noindent{\sf Case 2: $1\leq r(\b x) < M$.} In this situation, $q(\b x) + \left\lceil \frac{(k+1)r(\b x)+x_d}{(k+1)M}\right\rceil$ is equal to $q(\b x) + 1$ if $0\leq x_d\leq (k+1)(M-r(\b x))$, and is  equal to $q(\b x) + 2$, \ie if $(k+1)(M-r(\b x)) < x_d < (k+1)M$.
Thus, in this case, the innermost sum in \eqref{eq:Proof_hStarSylvesterSimplexDecompositionFormula} simplifies to:
\[\bigl((k+1)(M-r(\b x))+1\bigr)t^{q(\b x)+1}
+ \bigl((k+1)r(\b x)-1\bigr)t^{q(\b x)+2} \,.\qedhere\]
\end{proof}

\subsection{Affine symmetry of $h^\ast$-polynomials}
The goal of this section is to show a particular symmetry of the $h^\ast$-vector of $\Sylv$. The proof will rely on the  following elementary lemma which is a special case of reciprocity for reciprocal domains, see \cite[Propositions 8.2 \& 8.3]{Stanley1974}.
For sake of readability, we include a short proof.
See \cite[Section~4]{Beck_book} for a detailed presentation of Ehrhart--McDonald reciprocity and other tools.

\begin{lemma}\label{lem:reciprocity}
Let $\pol=\{\b x\in\mathbb{R}^d~|~\langle \b a_j, \b x\rangle\leq b_j, \text{ for } j\in[r]\}$ 
be a $d$-dimensional lattice polytope, where the $\b a_j\in \mathbb{Z}^d$ are primitive facet normals and $b_j\in \mathbb{Z}$ for $j\in [r]$.
Suppose that there exists 
$\b u\in \pol\cap\mathbb{Z}^d$ such that $b_j-\langle \b a_j,\b u\rangle\in\{0,1\}$ for every $j\in[r]$. Let $B$ be the union of  facets of $\pol$ containing $\b u$. Then, for every $m \geq 0$, we have
\[
 \Ehr(\pol\ssm B;m) ~=~ \Ehr(\pol^\circ;m+1) ~=~ (-1)^d\Ehr(\pol;-m-1) \,.
\]
Equivalently,
\[
h^\ast(\pol\ssm B;t) = t^d \cdot h^\ast(\pol;\tfrac{1}{t}).
\]
\end{lemma}

\begin{proof}

We consider the affine map $\phi_{\b u}:\mathbb{R}^d\to\mathbb{R}^d$ which is the translation by $\b u$, \ie $\phi_{\b u}(\b x)=\b x+\b u$ for $\b x\in \mathbb{R}^d$. We claim that $\phi_{\b u}$ induces a bijection from $ (m\pol\ssm mB)\cap\mathbb{Z}^d$ to $((m+1)\pol)^\circ\cap\mathbb{Z}^d$. Since $\phi_{\b u}$ is clearly injective, it suffices to show that $\phi_{\b u}((m\pol\ssm mB)\cap\mathbb{Z}^d)=((m+1)\pol)^\circ\cap\mathbb{Z}^d$. 
First, let $\b x\in (m\pol\ssm mB)\cap\mathbb{Z}^d$. Since all vectors involved are integral, we then have
\[
\langle \b a_j, \b x\rangle\leq 
 \begin{cases}
  mb_j-1,&\text{if } b_j-\langle \b a_j,\b u\rangle=0,\\
  mb_j,&\text{if } b_j-\langle \b a_j,\b u\rangle=1.
 \end{cases}
\]
This implies 
\[
\langle \b a_j, \b x+\b u\rangle=\langle \b a_j, \b x\rangle + \langle \b a_j, \b u\rangle\leq \begin{cases}
mb_j-1+b_j=(m+1)b_j-1, &\text{if } b_j-\langle \b a_j,\b u\rangle=0,\\
mb_j+ b_j-1=(m+1)b_j-1, &\text{if }  b_j-\langle \b a_j,\b u\rangle=1.
\end{cases}
\]
Hence, $\phi_u(\b x)=\b x+\b u\in ((m+1)\pol)^\circ\cap\mathbb{Z}^d$. 
The reverse containment follows by reversing the arguments. 
It follows that $ \Ehr(\pol\ssm B;m)=\Ehr(\pol^\circ;m+1)= (-1)^d\Ehr(\pol;-m-1)$, where the  second equality follows from classical Ehrhart--Macdonald reciprocity.
On the level of $h^\ast$-polynomials, using Ehrhart--Macdonald reciprocity again, this equality translates as follows
\[
\frac{h^\ast(\pol\ssm B;t)}{(1-t)^{d+1}}= \sum_{m\geq0}\Ehr(\pol\ssm B;m)t^m=\sum_{m\geq0}\Ehr(\pol^\circ;m+1)t^m
 =\frac{1}{t}\sum_{m\geq1}\Ehr(\pol^\circ;m)t^m
 =\frac{t^d\cdot h^\ast(\pol;\tfrac{1}{t})}{(1-t)^{d+1}} \,.
\]
This finishes the proof.
\end{proof}

We can apply the previous lemma to show the following symmetry.
\begin{proposition}\label{prop:SymmetryhStar}
We have 
\[
h^\ast(\Sylv[1];t)-h^\ast(\Sylv[0];t) = t^d\cdot h^\ast(\Sylv[0];\tfrac{1}{t}).
\]
\end{proposition}

\begin{proof}

Since $\Sylv[1]=\Sylv[0]\sqcup (\Sylv[1]\ssm\Sylv[0])$, it follows from additivity of the Ehrhart series that 
$h^\ast(\Sylv[1];t)-h^\ast(\Sylv[0];t)=h^\ast(\Sylv[1]\ssm \Sylv[0];t)$.
By 
\Cref{lem:UnimodularEquivalenceLayering}, the later is equal to the polynomial $h^\ast(\Sylv[0]\ssm \Base;t)$.

Again by \Cref{lem:UnimodularEquivalenceLayering}, the facet $\polF_d^{0}$ of $\Sylv[0]$ (\ie the facet opposing the vertex $\b 0$ in $\Sylv[0]$) is unimodular equivalent to $\Base$ which implies that $h^\ast(\Sylv[0]\ssm \Base;t)=h^\ast(\Sylv[0]\ssm \polF_d^{0};t)$. 
Consider the point $(1,\ldots,1)$. Using \Cref{prop:SylvesterSimplexBasicFacts}~\eqref{item:SylvInequalities}~\&~\eqref{item:SylvInteriorLatticePointsInFacets}, one easily verifies that $(1,\ldots,1)$ lies in the relative interior of $\polF_d^{0}$ and that it has lattice distance $1$ one from each other facet of~$\Sylv[0]$.
Combining the preceding identities and using \Cref{lem:reciprocity} we conclude that 
\[
h^\ast(\Sylv[1];t)-h^\ast(\Sylv[0];t) = h^\ast(\Sylv[0]\ssm \polF_d^{0};t) = t^d \cdot h^\ast(\Sylv[0];\tfrac{1}{t}) \,.\qedhere
\]
\end{proof}

\begin{corollary}\label{cor:Symmetry_h_star}
For $i\in [d-1]$, let $h^\ast_i(\Sylv) = a_{d,i}\cdot k + b_{d,i}$ be the coefficient of $h^\ast(\Sylv; t)$ on $t^i$.
Then: $a_{d,i} = b_{d, d-i}$.
\end{corollary}

\begin{proof}
By \Cref{thm:HaveAffineCoefficients}, each coefficient of $h^\ast(\Sylv;t)$ is indeed an affine function of $k$, in particular: $b_{d, i} = h^\ast_i(\Sylv[0])$ and $b_{d, i} + a_{d, i} = h^\ast_i(\Sylv[1])$.
With the notations of \Cref{prop:SymmetryhStar}, we get:
$h^\ast(\Sylv; t) = k\cdot h^\ast(\Sylv[d][1];t) - (k-1)\cdot h^\ast(\Sylv[d][0];t) = k\cdot\bigl(h^\ast(\Sylv[d][1];t) - h^\ast(\Sylv[d][0];t)\bigr) + h^\ast(\Sylv[d][0];t) = k\cdot t^d h^\ast(\Sylv[d][0];\tfrac{1}{t}) + h^\ast(\Sylv[d][0];t)$.
In particular, on the right-hand-side, the coefficient on $k$ in the coefficient of $t^i$ is $b_{d, d-i}$.
As this same coefficient is $a_{d, i}$ on the left-hand-side, this proves the claimed symmetry.
\end{proof}

\subsection{$h^\ast$-unimodality}
The symmetry of \Cref{cor:Symmetry_h_star} enables us to prove that the $h^\ast$-vector of $\Sylv$ is a unimodal sequence.
This can be verified by hand for $d\leq 6$ on \Cref{tab:hstar_coefficients}. 
To this end, we use this result from Athanasiadis (already proven by Hibi and Stanley in 1999, but unpublished).

\begin{quote}
{\cite[Theorem~1.3]{Athanasiadis2004-hStarVectorsAndEulerianPolynomials}} 
If a lattice polytope $\pol\subseteq\R^d$ has a regular and unimodular triangulation, then $h^\ast(\pol; t) = \sum_{k = 0}^d h^\ast_k t^k$ satisfies:
$h^\ast_{\left\lfloor\frac{d+1}{2}\right\rfloor} \geq \dots \geq h^\ast_{d-1} \geq h^\ast_d$.
 \label{item:AthanasiadisDecreasinghStar}
\end{quote}

We can now prove \Cref{thmC}.

\begin{theorem}\label{thm:hStarUnimodal}
For $d \geq 1$ and $k \geq 0$, the coefficients $h_i^\ast(\Sylv)$ form a unimodal sequence with mode at $\frac{d}{2}$ if $d$ is even, and with mode at $\frac{d-1}{2}$ or $\frac{d+1}{2}$ if $d$ is odd.
\end{theorem}

\begin{proof}
We will show the stronger statement that both sequences $(a_{d,0}, \dots, a_{d,d})$ and $(b_{d,0}, \dots, b_{d,d})$ defined in by $h^\ast_i(\Sylv) = a_{d, i}\cdot k + b_{d, i}$ are unimodal with mode  $\left\lceil\frac{d-1}{2}\right\rceil$ or $\left\lfloor\frac{d+1}{2}\right\rfloor$.
This will show the claim since the sum of (positively scaled) unimodal sequences whose modes differ by at most~$1$, is again unimodal. 

Since, by \Cref{thm:fruTriangulationSylvesterSimplices},  $\Sylv$ admits a regular,  unimodular triangulation for any $k\geq 0$, it follows from \cite[Theorem 1.3]{Athanasiadis2004-hStarVectorsAndEulerianPolynomials} that the second half of $h^\ast(\Sylv)$ (starting with $h^\ast_{\lfloor \frac{d+1}{2}\rfloor}(\Sylv)$) is decreasing.
 For $k=0$, this gives  $b_{d,i} \geq b_{d,i+1}$ for $i \geq \left\lfloor\frac{d+1}{2}\right\rfloor$. For $k\geq 1$, we get $a_{d,i}\cdot k + b_{d,i} \geq a_{d,i+1} \cdot k + b_{d,i+1}$ for $i\geq \left\lfloor\frac{d+1}{2}\right\rfloor$, \ie $a_{d,i} + \frac{b_{d,i}}{k} \geq a_{d,i+1} + \frac{b_{d,i+1}}{k}$ for $i\geq \left\lfloor\frac{d+1}{2}\right\rfloor$.
In particular, taking $k$ large enough (\eg $k > \max_i{b_{d,i}}$) ensures that $a_{d,i}\geq a_{d,i+1}$ for $i\geq \left\lfloor\frac{d+1}{2}\right\rfloor$.

Using \Cref{cor:Symmetry_h_star}, we obtain
\[
b_{d,i}=a_{d,d-i}\leq a_{d,d-i-1}=b_{d,i+1} \text{ for } 0\leq i\leq \left\lfloor \frac{d-2}{2}\right \rfloor
\]
and the analogous argument shows the $a_{d,i}\leq  a_{d,i+1}$ for $0\leq i\leq \left\lfloor\frac{d-2}{2}\right\rfloor$.
It follows that $(a_{d,0}, \dots, a_{d,d})$ and $(b_{d,0}, \dots, b_{d,d})$ are unimodal sequences with mode at  $\left\lceil\frac{d-1}{2}\right\rceil$ or $\left\lfloor\frac{d+1}{2}\right\rfloor$.
\end{proof}

\begin{remark}
Having a regular unimodular triangulation has strong consequences, as mentioned in \Cref{rmk:BekteMcMullen} and Athanasiadis' Theorem \cite[Theorem~1.3]{Athanasiadis2004-hStarVectorsAndEulerianPolynomials}.
However, recently, Ferroni \cite{ferroni2026unimodality} gave an example of a lattice polytope with a regular unimodular triangulation whose $h^\ast$-vector is not unimodal.
\end{remark}

\begin{remark}\label{rmk:fStarUnimodal}
It is natural to ask whether the $f^\ast$-vector of $\Sylv$ is unimodal. This is true for $1\leq d\leq 3$ due to \cite[Theorem 5]{BeckDeligeorgakiHlavacekValenciaPorras2024-fStarVectors} and for $d \leq 7$ and $k \geq 0$ using \Cref{tab:f_star_coefficients}.
However, we do not know whether this remains true in all dimensions since unimodality of the  
$h^\ast$-vector does not necessarily imply unimodality of the corresponding $f^\ast$-vector.
\end{remark}

\subsection{Number of lattice points}
Thanks to the closed-form expression for $h^\ast(\Sylv)$, we can also provide an explicit formula for  the number of lattice points contained in $\Sylv$.

\begin{theorem}\label{thm:NumberOfLatticePoints}
For $k\geq 0$:
$$\bigl|\Sylv \cap \Z^d\bigr| = d+1 + h_1^\ast(\Sylv) = d + (k+1)(s_d-1) +  \sum_{\substack{0\leq x_i<s_i\\ 1\leq N(\b x)\leq s_d-2}} \bigl((k+1)((s_d-1)-N(\b x))+1\bigr)\, .$$ 
\end{theorem}

\begin{proof}

The first equality is a general property of $h^\ast$ vectors discussed in Section \ref{ssec:ehrhart}.

For the second equality, we evaluate the formula from \Cref{thm:hStarSylvester} for $h_1^\ast$. It follows that $h_1^\ast(\Sylv)$ is the sum of the coefficients of $t$ of $\varphi_k(\b x; t)$ over all $\b x\in \Z^{d-1}$ such that $0\leq x_i < s_i$ for all $i\in [d-1]$. 
The coefficient of $t$ in $\varphi_k(\b x; t)$ is non-zero if and only if 
$(r(\b x) = 0$ and $q(\b x) \in \{0, 1\})$ or $(r(\b x) \ne 0$ and $q(\b x) = 0)$.

Since $N(\b x) = q(\b x) \cdot (s_d - 1) + r(\b x)$, the first case is equivalent to $N(\b x) \in\{0, s_d - 1\}$.
We have $N(\b x) = 0$ if and only if $\b x = \b 0$, and the coefficient of $t$ in $\varphi_k(\b 0; t)$ is equal to $(k+1)\cdot(s_d-1) - 1$.
On the other hand, %
if $N(\b x) = \sum_{i=1}^{d-1}x_i\frac{s_d - 1}{s_i}=s_d - 1$, then, since $s_j$ divides $s_d-1$ for every $1\leq j\leq d-1$ and $\frac{s_d - 1}{s_i}$ if $i\neq j$, it follows that $s_j$ also  divides $x_j\frac{s_d - 1}{s_j}$ for every $1\leq j\leq d-1$.  %
As $\gcd\!\left(\frac{s_d - 1}{s_j}, s_j\right) = 1$, this implies that $s_j$ divides $x_j$ and using that $0\leq x_j < s_j$, we conclude that $x_j = 0$ for all $j$, contradicting $N(\b x) = s_d - 1 \ne 0$.

To handle the second case, note  that $r(\b x) \ne 0$ and $q(\b x) = 0$ if and only if $N(\b x) \in \{1, \dots, s_d - 2\}$.
In this case, the coefficient of $t$ in $\varphi_k(\b x; t)$ is equal to $(k+1)\cdot(s_d - 1 - N(\b x)) + 1$.

The formula for $h_1^\ast(\Sylv)$ (and consequently $\bigl|\Sylv\cap\mathbb{Z}^d\bigr|$) follows by combining both cases.
\end{proof}

Note that \Cref{conj:hstar} states that $h^\ast_1(\pol) \leq h_1^\ast(\Sylv)$, where $\pol$ is any lattice $d$-polytope with exactly $k$ interior lattice points.
The significance of \Cref{thm:NumberOfLatticePoints} is that it provides an explicit upper bound that can be directly compared with $h_1^\ast(\pol)$.

To complete the study of Ehrhart-theoretic properties of $\Sylv$, we also investigate  its local $h^\ast$- and boundary $h^\ast$-polynomials.
However, since the ideas behind their proofs are similar to those in \Cref{prop:hStarDiagonalSimplices} and \Cref{thm:hStarSylvester}, we decided to include these results in  Appendix~\ref{Appen:Local_boundary}.

\section{Magic positivity for small dimensions but not for \texorpdfstring{$d = 7$}{d = 7}}\label{sec:MagicPositivity}

Two natural problems are to determine when a lattice polytope $\pol$ is Ehrhart-positive and to study algebraic properties of its $h^\ast$-polynomial such as real-rootedness. 
In general, these two properties are difficult to tackle.
Yet, the following notion allows us to treat them simultaneously.

We write the Ehrhart polynomial $\Ehr(\pol; t)$ of a $d$-dimensional lattice polytope $\pol$ in the so-called \defn{magic basis}
$\{t^i(t+1)^{d-i} ~|~ 0 \leq i \leq d \}$ as follows
$$
\Ehr(\pol; t) = \sum_{i = 0}^d c_i\, t^i(t + 1)^{d-i} \,.
$$
If $c_i \geq 0$ for every $i = 0,\ldots,d$, then we say that $\Ehr(\pol; t)$ is \defn{magic positive}.
The relevance of this notion comes from the following result.
Our terminology follows \cite[Theorem~4.19]{FerroniHigashitani2024ExamplesCounterexamplesEhrhartTheory}, but this was proven in \cite[Section~4]{Magic_Petter}.

\begin{quote}
For $\pol$ a lattice polytope, if $\Ehr(\pol ; t)$ is magic positive, then $\Ehr(\pol; t)$ has non-negative coefficients, and $h^\ast(\pol; t)$ is real-rooted.
\end{quote}

First,  using our triangulation of \Cref{ssec:UnimodularLayering}, it is straightforward to prove that the coefficients in the magic basis are affine functions of $k$.

\begin{lemma}\label{lem:MagicCoeffAreAffine}
For any $d, k\geq 0$, writing
$\Ehr(\Sylv; t) = \sum_{i=0}^d c_i(k)\cdot t^i(t+1)^{d-i}$, we have that each $c_i(k)$ is an affine function of $k$.
\end{lemma}

\begin{proof}
By \Cref{thm:HaveAffineCoefficients}, the Ehrhart polynomial $\Ehr(\Sylv; t)$ is an affine function of $k$.
Since the change of basis from $\set{t^j}{0\leq j\leq d}$ to $\set{t^i(t+1)^{d-i}}{0\leq i\leq d}$ is linear and does not depend on $k$, the coefficients $c_i(k)$ in the magic basis are also affine functions of $k$.
\end{proof}

\Cref{lem:MagicCoeffAreAffine} directly shows that to compute the  coefficients of $\Ehr(\Sylv,t)$ in the magic basis for a fixed dimension $d\geq 1$ and any $k$, it suffices to do the exact computations for $k\in\{0,1\}$. 
We recall and  prove the first part of \Cref{thmD}.

\begin{corollary}\label{coro:magic_positive}
The Ehrhart polynomial $\Ehr(\Sylv;t)$ is magic positive for all $k\geq 0$ and every $1\leq d\leq 6$.
Consequently, $\Sylv$ is Ehrhart positive and $h^\ast$-real-rooted for all $k\geq 0$ and $1 \leq d \leq 6$.
\end{corollary}

\begin{proof}
The coefficients $c_i(k)$ are listed in \Cref{tab:magic_coefficients} for $1 \leq d\leq 6$.
They are positive for all $k \geq 0$.
The second part follows from \cite[Section~4]{Magic_Petter}.
\end{proof}

Going further, the symmetry that we proved for the coefficients of the $h^\ast$-polynomial in \Cref{cor:Symmetry_h_star} also holds for the coefficients in the magic basis.%

\begin{proposition}\label{prop:Symmetry_MagicEhrhart}
For $d, k\geq 0$, let $c_i(k) = p_{d,i} \cdot k + q_{d,i}$ be the coefficients of $\Ehr(\Sylv;t)$ in the magic basis.
Then $p_{d,i} = q_{d,d-i}$ for every $0 \leq i \leq d$.
\end{proposition}

\begin{proof}
Let 
\[
B_j(t)=\binom{t+d-j}{d}\,, \text{ and } M_i(t)=t^i(t+1)^{d-i}\,.
\]
We use $\b T_d = (T_{i,j})_{0 \leq i, j\leq d}$ to denote the basis exchange matrix from $B_j(t)$ to $M_i(t)$, \ie $B_j(t)=\sum_{i=0}^d T_{i,j}M_i(t)$.
Consider the involution $\tau$ on polynomials of degree at most $d$ defined by $\tau(f)(t) = (-1)^d \cdot f(-t-1)$.
Then 
\begin{align*}
\tau(B_j)(t)&=(-1)^d\binom{-(j+t+1-d)}{d}=\binom{t+j}{d}=B_{d-j}(t) \,,\\
\tau(M_i)(t)&=t^{d-i}(t+1)^i=M_{d-i}(t) \,.
\end{align*}
Applying $\tau$ to $B_j(t)=\sum_{i=0}^d T_{i,j}M_i(t)$ gives $B_{d-j}(t)=\sum_{i=0}^d T_{i,j}M_{d-i}(t)$.
Comparing this with the expansion of $B_{d-j}(t)$ in the magic basis yields $T_{i,j}=T_{d-i,d-j}$.
By \Cref{cor:Symmetry_h_star}, we have $h_j^\ast(\Sylv) = b_{d,j} + b_{d,d-j}\cdot k$ for every $0\leq j \leq d$.
This implies $$c_i(k) = k\cdot\sum_{j=0}^d T_{i,j}b_{d,d-j} ~+~ \sum_{j=0}^d T_{i,j}b_{d,j} \,.$$

So $p_{d,i}=\sum_{j=0}^d T_{i,j}b_{d,d-j}$ and $q_{d,i}=\sum_{j=0}^d T_{i,j}b_{d,j}$.
Using $T_{i,d-\ell}=T_{d-i,\ell}$, the claim follows.
\end{proof}

For a fixed $d\geq 1$, as the magic coefficients are affine in $k$ (\Cref{lem:MagicCoeffAreAffine}), and exhibit an affine symmetry (\Cref{prop:Symmetry_MagicEhrhart}), it is enough to determine only $d$ coefficients in order to conclude whether $\Sylv$ is magic positive or not for all $k\geq 0$.

\begin{corollary}\label{cor:MagicPositivityIsImpliesByTruthAtK0}
For any $d\geq 1$, if $\Ehr(\Sylv[0]; t)$ is magic positive, then $\Ehr(\Sylv; t)$ is magic positive for all $k\geq 0$.

Moreover, for any $d\geq 1$, if $\Ehr(\Sylv[0]; t)$ is not magic positive, then there exists $K>0$ such that $\Ehr(\Sylv; t)$ is not magic positive for all $k\geq K$.
\end{corollary}

\begin{proof}
The first statement is a direct consequence of \Cref{lem:MagicCoeffAreAffine,prop:Symmetry_MagicEhrhart}. 

Conversely, assume $c_i(0) = q_{d, i} < 0$ for some $i\in [d]$. %
Using \Cref{prop:Symmetry_MagicEhrhart} we conclude
\[
c_{d-i}(k) = q_{d,d-i} + k\cdot p_{d,d-i}
= q_{d,d-i} + k\cdot q_{d,i}.
\]
The expression on the right-hand-side is strictly negative for all $k > \frac{-q_{d,d-i}}{q_{d,i}}$.
\end{proof}

Using an improvement of our implementation for the computation of $h^\ast(\Sylv; t)$ suggested by ChatGPT using GPT-5.6 Sol, we ran a computational experiment yielding the values for the $7$-dimensional case $\Ehr(\Sylv[k][7]; t)$, see \Cref{ssec:TablesD7}.
In particular, we obtain that the linear coefficient of $\Ehr(\Sylv[k][7]; t)$ is negative, namely it is:
$$- \Bigl(\dfrac{6216272723965441907549}{30}\,k  ~+~ \dfrac{12432545447927574463817}{60}\Bigr)$$

This yields the second part of \Cref{thmD}.

\begin{proposition}\label{prop:NotMagicPositive}
For $k\geq 0$, the simplex $\Sylv[k][7]$ is neither Ehrhart positive nor magic positive.
\end{proposition}

\begin{proof}
This is a direct consequence of the negativity of the linear coefficient of $\Ehr(\Sylv[k][7]; t)$. 
\end{proof}

Using ChatGPT with GPT-5.6 Sol, we computed the Ehrhart polynomials of $\Sylv[0][d]$ up to $d=16$. Based on the results, we conjecture that $\Ehr(\Sylv;t)$ has a negative coefficient in front of $t^i$ if and only if $1\leq i\leq d-6$ and,  $d-i\equiv 2 \mod 4$ or $d-i\equiv 3 \mod 4$.
In particular, to motivate further research, we state
\begin{conjecture}\label{conj:MagicPositive}
For $d\geq 7$ and any $k\geq 0$, the Ehrhart polynomial $\Ehr(\Sylv; t)$ is not Ehrhart positive. %
\end{conjecture}

Note that $\Sylv[k][7]$ is $h^\ast$-unimodal (as proven in \Cref{thm:hStarUnimodal}) and that $\Sylv[0][7]$ is $h^\ast$-real-rooted (direct computation from \Cref{ssec:TablesD7}).
This encourages us to conjecture the following:

\begin{conjecture}\label{conj:hStarRealRooted}
The polynomial $h^\ast(\Sylv; t)$ is real-rooted for every $k\geq 0$ and $d \geq 1$.
\end{conjecture}

Using ChatGPT with GPT-5.6 Sol, we verified this conjecture for $d\leq 16$.
We have tried to tackle these conjectures using the partial recursive formulas that we derived for the polynomial $f^\ast(\Sylv; t)$ in \Cref{thm:f_star_induction} and the values of $\Ehr(\Sylv[k][7]; t)$ and $h^\ast(\Sylv[k][7]; t)$, but could not find a complete proof.
\appendix
\addtocontents{toc}{\protect\setcounter{tocdepth}{1}}

\section{Theoretic results on boundary and local $h^\ast$-polynomials}\label{Appen:Local_boundary}

\subsection{Local $h^\ast$-polynomial}
Let $\polS$ be a simplex with vertices $\b v_0, \dots, \b v_d$.
Analogous to the half open parallelepiped of \Cref{eq:half_open_para}, the \defn{open parallelepiped} associated to $\polS$ is defined as 
$$\mathdefn{\overline{\Pi}_{\polS}} \coloneqq \set{\sum_{j=0}^d \lambda_j \begin{pmatrix} \b v_j \\ 1\end{pmatrix}}{0< \lambda_j < 1 \text{ for } 0\leq j\leq d} \,.$$ 
The \mathdefn{local $h^\ast$-polynomial} \cite{Local_h_definition} (or \emph{box polynomial} \cite{box_polynomial_definition}) 
is defined by $\mathdefn{\ell^\ast(\polS; t)} \coloneqq \sum_{(\b x, m)\in \overline{\Pi}_{\polS}\cap \Z^{d+1}} t^m$, so that the coefficient of $t^i$ counts the number of lattice points in $\overline{\Pi}_{\polS}$ whose last coordinate is equal to $i$.

On the other hand, the union of all facets of $\pol$ forms the \defn{boundary} of $\pol$, denoted by $\mathdefn{\partial \pol}$, while the \defn{interior} of $\pol$ is $\mathdefn{\pol^{\circ}} \coloneqq \pol\ssm \partial \pol$.

\begin{definition}
We define the \mathdefn{boundary $h^\ast$-polynomial} and the \mathdefn{interior $h^\ast$-polynomial} of $\pol$, denoted by $h^\ast(\partial \pol; t)$ and $h^\ast(\pol^\circ; t)$, respectively, via
\[
	\sum_{m\geq 0} \Ehr(\partial \pol;m)\cdot t^m =\frac{h^\ast(\partial \pol; t)}{(1-t)^d} \quad \text{and}\quad \sum_{m\geq 1} \Ehr(\pol^\circ;m)\cdot t^m =\frac{h^\ast(\pol^\circ; t)}{(1-t)^{d+1}},
\]
where
\[
	\Ehr(\partial \pol; m) \coloneqq \bigl|\partial(m\pol) \cap \mathbb{Z}^d \bigr| \qquad \text{and} \qquad \Ehr(\pol^\circ; m) \coloneqq \bigl|(m\pol)^\circ \cap \mathbb{Z}^d \bigr|.
\]
\end{definition}

It is well known that $h^\ast(\partial \pol; t)$ is palindromic, that is, $h^\ast(\partial \pol; t) = t^d\cdot h^\ast(\partial \pol; \frac{1}{t})$.
For a comprehensive study of the properties of $h^\ast(\partial \pol; t)$, we refer the reader to \cite{BajoBeck2023-BoundaryhStar, boundaryhunimodulartriangulations}.

\medskip

Let $\polD$ be a simplex with vertices $\b 0, a_1\b e_1 \dots, a_d\b e_d$.
Recall that
$$\mathdefn{\overline{\Pi}_{\polytope{D}}} \coloneqq \set{ \lambda_0 \begin{pmatrix} \b 0 \\ 1\end{pmatrix} + \sum_{j=1}^d \lambda_j \begin{pmatrix} a_j \b e_j \\ 1\end{pmatrix}}{0 < \lambda_j < 1}$$
is the open parallelepiped associated to the diagonal $d$-simplex $\polytope{D} = \conv(\b 0, a_1\b e_1, \dots, a_d\b e_d)$.
Then its local $h^\ast$-polynomial is given by:
$$\ell^\ast(\polD; t) \coloneqq \sum_{(\b x, m)\in \overline{\Pi}_{\polytope{D}}\cap \Z^{d+1}} t^m \,.$$

Using an argument analogous to \Cref{prop:hStarDiagonalSimplices} (with $m = \lambda_0 + \rho(\b x)$ and $0 < \lambda_0 < 1$), one can prove the following equation:
\begin{equation}\label{eqn:local_h_star}
\ell^\ast(\polD; t) = \sum_{\substack{\b x\in\Z^{d}\\0 < x_j < a_j \text{ and } \rho(\b x) \notin \Z}} t^{\left\lceil  \rho(\b x)\right \rceil},  \qquad \text{ where } \mathdefn{\rho(\b x)} \coloneqq \sum_{j = 1}^d \frac{x_j}{a_j}\,.
\end{equation}
The details are omitted and left to the reader.
We prove the analogous version of \Cref{thm:hStarSylvester} for the local $h^\ast$-polynomial. 
This is the formula we implemented to obtain the $5$-dimensional and $6$-dimensional results in \Cref{tab:local_h_coefficients}.

\begin{proposition}\label{prop:local_h_star_symmetry}
For $k\geq 0$, we have
$$\ell^\ast (\Sylv ; t) = \sum_{\substack{\b x \in \Z^{d-1}\\ 0 < x_j <s_j }} \bigl((k+1)(s_d-1 - r(\b x)) - 1\bigr) t^{q(\b x) + 1} + \bigl((k+1)r(\b x) - 1\bigr) t^{q(\b x) + 2}.$$
\end{proposition}

\begin{proof}
Let $M = s_d - 1$.
By \Cref{eqn:local_h_star}, we have
$$\ell^\ast(\Sylv; t) = \sum_{\substack{\b x\in\Z^{d}\\0<x_j<s_j}} t^{\left\lceil \sum\limits_{j = 1}^{d - 1} \tfrac{x_j}{s_j} + \tfrac{x_d}{(k+1)M} \right\rceil} \,,$$
where we sum over all $\b x$ satisfying $\sum\limits_{j = 1}^{d - 1} \tfrac{x_j}{s_j} + \tfrac{x_d}{(k+1)M}\notin \Z$.
If $\b x\in\Z^{d-1}$ with $0 < x_j < s_j$ for $j\in [d-1]$, satisfies the latter condition, then the  proof of \Cref{thm:hStarSylvester} shows that  $\left \lceil \frac{1}{M}N(\b x) + \frac{x_d}{(k+1)M} \right \rceil = q(\b x) + \left\lceil\frac{(k+1)r(\b x) + x_d}{(k+1)M}  \right \rceil$.

Moreover, since $\frac{M}{s_j} = (s_j - 1)\cdot s_{j+1}\cdots s_{d-1} \equiv s_j - 1 \equiv -1 \pmod {s_j}$, we have 
$$r(\b x)\equiv N(\b x) \equiv x_j\dfrac{M}{s_j}\equiv -x_j \pmod{s_j}.$$
As $x_j \ne 0$, it follows that  $r(\b x) \ne 0$ and hence, $1\leq r(\b x) < M$.
Since $1\leq x_d <(k+1)M$, we get that $0<\dfrac{(k+1)r(\b x) + x_d}{(k+1)M}<2$.
Note that if $(k+1)r(\b x) + x_d = (k+1)M$, then $\frac{1}{M}N(\b x) + \frac{x_d}{(k+1)M}\in\Z$, contradicting the hypothesis on $\b x$.
Thus
$$\left\lceil\frac{(k+1)r(\b x) + x_d}{(k+1)M} \right \rceil =\begin{cases}
1, & \text{if } 0<x_d <(k+1)(M-r(\b x)) \,, \\
2, & \text{if } (k+1)(M-r(\b x))<x_d <(k+1)M \,.
\end{cases}$$
This implies the claimed result.
\end{proof}

\subsection{Boundary $h^\ast$-polynomial}\label{subsec:boundary}
It is well-known (see \eg \cite{BajoBeck2023-BoundaryhStar}) that the boundary $h^\ast$-polynomial is closely related to the symmetric decomposition of the ordinary $h^\ast$-polynomial. More precisely, it can be computed as follows:

\begin{equation}\label{eq:boundary_h_star}
h^\ast(\partial \Sylv; t ) = \frac{h^\ast(\Sylv ; t) - h^\ast(\Sylv[k][d]; \frac{1}{t})t^{d+1}}{1-t}
\end{equation}

Based on \Cref{thm:hStarSylvester,prop:local_h_star_symmetry}, we now provide an explicit formula for $h^\ast(\partial \Sylv; t)$:

\begin{proposition}\label{prop:boundaryhStarSylvester}
For all $d\geq 2$ and $k\geq 0$:
$$h^\ast(\partial \Sylv; t) = \sum_{\substack{\b x\in \Z^{d-1} \\ 0\leq x_j < s_j}} \psi_k(\b x;t) ~,$$
where
$$ 
\mathdefn{\psi_k(\b x;t)} \coloneqq
\begin{cases}
t^{q(\b x)}+\bigl((k+1)(s_d-1)-k\bigr)t^{q(\b x)+1}, & \text{if } r(\b x)=0, \\
\bigl((k+1)(s_d-1 - r(\b x))+1\bigr)t^{q(\b x)+1} + \bigl((k+1)r(\b x)-k\bigr)t^{q(\b x)+2}, & \text{if } 1\leq r(\b x)<s_d - 1.
\end{cases}
$$
\end{proposition}

\begin{proof}
Let $h^\ast(\Sylv;t) = \sum_{i=0}^d (a_{d,i}\cdot k + b_{d,i})t^i = A_d(t)\cdot k + B_d(t)$.
By \Cref{cor:Symmetry_h_star}, we have $h^\ast_i(\Sylv) = b_{d,d-i}\cdot k + b_{d,i}$.
Therefore, $t^dB_{d}\bigl(\frac{1}{t}\bigr) = A_{d}(t)$, and $t^dA_{d}\bigl(\frac{1}{t}\bigr) = B_d(t)$.
 \Cref{eq:boundary_h_star} implies
\begin{align*}
h^\ast(\partial \Sylv; t )
&= \frac{A_d(t)\cdot k + B_d(t) - \bigl(t^{d+1}A_d(\tfrac1t)\cdot k + t^{d+1}B(\tfrac1t)\bigr)}{1-t} \\
&= \dfrac{B_d(t) + A_d(t)\cdot k - (A_d(t) + B_d(t) k) t}{1-t} = A_d(t) + B_d(t) + \dfrac{k-1}{1-t}\bigl(A_d(t) - tB_d(t)\bigr) \,.
\end{align*}

Furthermore, by \Cref{thm:hStarSylvester}, we have
$$h^\ast(\Sylv; t) = \sum_{\substack{\b x\in \Z^{d-1} \\ 0\leq x_j < s_j}} \varphi_k(\b x; t) = \sum_{\substack{\b x\in \Z^{d-1} \\ 0\leq x_j < s_j}} \bigl(\beta(\b x;t) + \alpha(\b x;t) \cdot k\bigr) = B_d(t) + A_d(t)\cdot k\,.$$
To ease notation, let $M \coloneqq s_d-1$.
Fix $\b x\in \Z^{d-1}$ with $0\leq x_j < s_j$ for all $j\in [d-1]$.
First, consider the case $r(\b x) = 0$.
By \Cref{thm:hStarSylvester}, the contribution corresponding to $\b x$ is $\varphi_k(\b x; t) = t^{q(\b x)}+\bigl((k+1)\cdot M-1\bigr)t^{q(\b x)+1}.$
Therefore, $\beta(\b x;t) = t^{q(\b x)} + (M-1)t^{q(\b x) + 1}$ and $\alpha(\b x;t) =Mt^{q(\b x)+1}$.
Moreover, $\beta(\b x;t) + \alpha(\b x; t) = t^{q(\b x)} + (2M-1)t^{q(\b x ) + 1}$ and $\alpha(\b x; t) - t \beta(\b x; t) =(1-t)\cdot(M-1) t^{q(\b x) + 1}$.
Thus, the contribution of $\b x$ to $h^\ast(\partial \Sylv; t )$ is 
$$\alpha(\b x; t) + \beta(\b x;t) + \dfrac{k-1}{1-t}\bigl(\alpha(\b x; t) - t\beta(\b x;t)\bigr) = t^{q(\b x)} + \bigl((k+1)M-k\bigr)t^{q(\b x)+1}\,.$$

Now consider the case $1\leq r(\b x)<M$.
By \Cref{thm:hStarSylvester}, the contribution corresponding to $\b x$ is
$\varphi_k(\b x; t) =\bigl((k+1)(M-r(\b x))+1\bigr)t^{q(\b x)+1} + \bigl((k+1)r(\b x)-1\bigr)t^{q(\b x)+2}$.
Therefore, $\beta(\b x;t) = \bigl(M-r(\b x)+1\bigr)t^{q(\b x)+1} + \bigl(r(\b x)-1\bigr)t^{q(\b x)+2}$ and $\alpha(\b x;t) = \bigl(M-r(\b x)\bigr)t^{q(\b x)+1} + r(\b x)t^{q(\b x)+2}$.
Hence, the contribution of $\b x$ to $h^\ast(\partial \Sylv;t)$ is 
$$\alpha(\b x; t) + \beta(\b x;t) + \dfrac{k-1}{1-t}\bigl(\alpha(\b x; t) - t\beta(\b x;t)\bigr) = \bigl((k+1)(M-r(\b x))+1  \bigr)t^{q(\b x)+1} + \bigl( (k+1)r(\b x) - k \bigr)t^{q(\b x) + 2} \,.$$
As these are the only two cases for $r(\b x)$, summing the contributions yields the desired formula.
\end{proof}

\section{Ehrhart-theoretic quantities for dimensions 7 and lower}\label{app:Tables}

One can consult the code to generate the numbers in the following tables in the NoteBook in SageMath available at
\url{https://github.com/EmbrunForestier/SylvesterSimplices}.
Note that the provided implementation is the version proposed by ChatGPT with GPT-5.6 Sol.

\subsection{Tables for dimensions 6 and lower}
\textcolor{white}{Nothing}

\begin{table}[htbp]
\setlength{\tabcolsep}{3.5pt}

\begin{adjustbox}{center,max width=0.96\textheight,max totalheight=0.70\textwidth}
\begin{tabular}{@{}c c @{}}
\toprule
 & $\bigl|\Sylv\cap \Z^d\bigr|$ \\
\midrule

$\Sylv[k][0]$
& $0k + 1$ \\
\addlinespace[2pt]

$\Sylv[k][1]$
& $1 k + 2$ \\
\addlinespace[2pt]

$\Sylv[k][2]$
& $3 k + 6$ \\
\addlinespace[2pt]

$\Sylv[k][3]$
& $16 k + \textcolor{blue}{\textbf{23}}$ \\
\addlinespace[2pt]

$\Sylv[k][4]$
& $\textcolor{blue}{\textbf{328}} k + \textcolor{red}{\textbf{352}}$ \\
\addlinespace[2pt]

$\Sylv[k][5]$
& $177716 k + 178069$ \\
\addlinespace[2pt]

$\Sylv[k][6]$
& $134562679888 k + 134562857958$ \\

\bottomrule
\end{tabular}
\end{adjustbox}
\caption{Number of lattice points of $\Sylv$ for $1 \leq d \leq 6$.
As an illustration of \Cref{cor:NbrLatticePoints_RecursionFormula}, the sum of the \textcolor{blue}{\textbf{blue}} numbers plus $1$ is equal to the \textcolor{red}{\textbf{red}} number.
}
\label{tab:LatticePointsNbr}
\end{table}
\begin{table}[htbp]
\setlength{\tabcolsep}{3.5pt}

\begin{adjustbox}{center,max width=0.96\textheight,max totalheight=0.70\textwidth}
\begin{tabular}{@{}c c c c c c c c@{}}
\toprule
 & $t^0$ & $t^1$ & $t^2$ & $t^3$ & $t^4$ & $t^5$ & $t^6$ \\
\midrule

$\Sylv[k][0]$
& $1$
& 
&
&
&
&
& \\
\addlinespace[2pt]

$\Sylv[k][1]$
& $1$
& $1k+1$
&
&
&
&
& \\
\addlinespace[2pt]

$\Sylv[k][2]$
& $1$
& $1k+3$
& $2k+2$
&
&
&
& \\
\addlinespace[3pt]

$\Sylv[k][3]$
& $1$
& $\frac{5}{2}k+\frac{11}{2}$
& $\frac{15}{2}k+\frac{21}{2}$
& $6k+6$
&
&
& \\
\addlinespace[5pt]

$\Sylv[k][4]$
& $1$
& $20k+\frac{51}{2}$
& $91k+\frac{203}{2}$
& $\frac{287}{2}k+\frac{301}{2}$
& $\frac{147}{2}k+\frac{147}{2}$
&
& \\
\addlinespace[8pt]

$\Sylv[k][5]$
& \makecell{\textcolor{white}{$0k$}\\[4pt]$1$}
& \makecell{$\frac{49129}{20}k$\\[4pt]$+\frac{49639}{20}$}
& \makecell{$\frac{167555}{8}k$\\[4pt]$+\frac{168361}{8}$}
& \makecell{$\frac{236887}{4}k$\\[4pt]$+\frac{237489}{4}$}
& \makecell{$\frac{543305}{8}k$\\[4pt]$+\frac{543907}{8}$}
& \makecell{$\frac{271803}{10}k$\\[4pt]$+\frac{271803}{10}$}
& \\
\addlinespace[8pt]

$\Sylv[k][6]$
& \makecell{\textcolor{white}{$0k$}\\[4pt]$1$}
& \makecell{$\frac{275766851}{5}k$\\[4pt]$+\frac{3309351281}{60}$}
& \makecell{$\frac{245086592963}{60}k$\\[4pt]$+\frac{490175711341}{120}$}
& \makecell{$\frac{182807811789}{8}k$\\[4pt]$+\frac{548424859903}{24}$}
& \makecell{$\frac{145214466488}{3}k$\\[4pt]$+\frac{1161717363625}{24}$}
& \makecell{$\frac{1775008403987}{40}k$\\[4pt]$+\frac{1775009491801}{40}$}
& \makecell{$\frac{295834824649}{20}k$\\[4pt]$+\frac{295834824649}{20}$}
\\
\bottomrule
\end{tabular}
\end{adjustbox}
\caption{Ehrhart polynomial of $\Sylv$, \ie $\Ehr(\Sylv; t)$, for $1 \leq d \leq 6$.}
\label{tab:ehrhart_coefficients}
\end{table}
\begin{table}[htbp]
\setlength{\tabcolsep}{3.5pt}

\begin{adjustbox}{center,max width=0.96\textheight,max totalheight=0.70\textwidth}
\begin{tabular}{@{}c c c c c c c c@{}}
\toprule
 & $t^0(1+t)^{d-0}$ & $t^1(1+t)^{d-1}$ & $t^2(1+t)^{d-2}$ & $t^3(1+t)^{d-3}$ & $t^4(1+t)^{d-4}$ & $t^5(1+t)^{d-5}$ & $t^6(1+t)^{d-6}$ \\
\midrule

$\Sylv[k][0]$
& $1$
& 
& 
& 
& 
& 
&  \\
\addlinespace[2pt]

$\Sylv[k][1]$
& $1$
& $k$
& 
& 
& 
& 
&  \\
\addlinespace[2pt]

$\Sylv[k][2]$
& $1$
& $1k+1$
& $k$
& 
& 
& 
&  \\
\addlinespace[3pt]

$\Sylv[k][3]$
& $1$
& $\frac{5}{2}k+\frac{5}{2}$
& $\frac{5}{2}k+\frac{5}{2}$
& $k$
& 
& 
&  \\
\addlinespace[5pt]

$\Sylv[k][4]$
& $1$
& $20k+\frac{43}{2}$
& $31k+31$
& $\frac{43}{2}k+20$
& $k$
& 
&  \\
\addlinespace[8pt]

$\Sylv[k][5]$
& \makecell{\textcolor{white}{$0k$}\\[4pt]$1$}
& \makecell{$\frac{49129}{20}k$\\[4pt]$+\frac{49539}{20}$}
& \makecell{$\frac{444743}{40}k$\\[4pt]$+\frac{445093}{40}$}
& \makecell{$\frac{445093}{40}k$\\[4pt]$+\frac{444743}{40}$}
& \makecell{$\frac{49539}{20}k$\\[4pt]$+\frac{49129}{20}$}
& \makecell{$k$\\[4pt]\textcolor{white}{$0$}}
&  \\
\addlinespace[8pt]

$\Sylv[k][6]$
& \makecell{\textcolor{white}{$0k$}\\[4pt]$1$}
& \makecell{$\frac{275766851}{5}k$\\[4pt]$+\frac{3309350921}{60}$}
& \makecell{$\frac{228540581903}{60}k$\\[4pt]$+$\textcolor{blue}{\textbf{$\frac{457082200331}{120}$}}}
& \makecell{$\frac{847608477371}{120}k$\\[4pt]$+\frac{847608477371}{120}$}
& \makecell{\textcolor{blue}{\textbf{$\frac{457082200331}{120}$}}$k$\\[4pt]$+\frac{228540581903}{60}$}
& \makecell{$\frac{3309350921}{60}k$\\[4pt]$+\frac{275766851}{5}$}
& \makecell{$k$\\[4pt]\textcolor{white}{$0$}} \\
\bottomrule
\end{tabular}
\end{adjustbox}
\caption{Magic coefficients of $\Ehr(\Sylv; t)$ for $0 \leq d \leq 6$.
Illustrating \Cref{prop:Symmetry_MagicEhrhart}, in the row for $\Sylv[k][6]$, reading the top left-to-right, or the bottom right-to-left give the same sequence.
}
\label{tab:magic_coefficients}
\end{table}


\begin{table}[htbp]
\setlength{\tabcolsep}{3.5pt}

\begin{adjustbox}{center,max width=0.96\textheight,max totalheight=0.70\textwidth}
\begin{tabular}{@{}c c c c c c c c@{}}
\toprule
 & $t^0$ & $t^1$ & $t^2$ & $t^3$ & $t^4$ & $t^5$ & $t^6$ \\
\midrule

$\Sylv[k][0]$
& $1$
& 
&
&
&
&
& \\
\addlinespace[2pt]

$\Sylv[k][1]$
& $1$
& $k$
&
&
&
&
& \\
\addlinespace[2pt]

$\Sylv[k][2]$
& $1$
& $3k+3$
& $k$
&
&
&
& \\
\addlinespace[3pt]

$\Sylv[k][3]$
& $1$
& $16k+19$
& $19k+16$
& $k$
&
&
& \\
\addlinespace[5pt]

$\Sylv[k][4]$
& $1$
& $328k+347$
& $1088k+1088$
& $347k+328$
& $k$
&
& \\
\addlinespace[8pt]

$\Sylv[k][5]$
& \makecell{\textcolor{white}{$0k$}\\[4pt]$1$}
& \makecell{$177716k$\\[4pt]$+178063$}
& \makecell{$1452548k$\\[4pt]$+1453308$}
& \makecell{$1453308k$\\[4pt]$+1452548$}
& \makecell{$178063k$\\[4pt]$+177716$}
& \makecell{$k$\\[4pt]\textcolor{white}{$0$}}
& \\
\addlinespace[8pt]

$\Sylv[k][6]$
& \makecell{\textcolor{white}{$0k$}\\[4pt]$1$}
& \makecell{$134562679888k$\\[4pt]$+134562857951$}
& \makecell{$\textcolor{blue}{\textbf{2398473782180}}k$\\[4pt]$+2398475057772 $}
& \makecell{$5583979309572k$\\[4pt]$+5583979309572$}
& \makecell{$2398475057772k$\\[4pt]$+\textcolor{blue}{\textbf{2398473782180}}$}
& \makecell{$134562857951k$\\[4pt]$+134562679888$}
& \makecell{$k$\\[4pt]\textcolor{white}{$0$}}
\\
\bottomrule
\end{tabular}
\end{adjustbox}
\caption{
$h^\ast$-polynomial of $\Sylv$ for $0 \leq d \leq 6$.
Illustrating \Cref{cor:Symmetry_h_star}, in the row for $\Sylv[k][6]$, reading the top left-to-right, or the bottom right-to-left give the same sequence.
}
\label{tab:hstar_coefficients}
\end{table}
\begin{table}[htbp]
\setlength{\tabcolsep}{3.5pt}

\begin{adjustbox}{center,max width=0.96\textheight,max totalheight=0.70\textwidth}
\begin{tabular}{@{}c c c c c c c c@{}}
\toprule
 & $t^0$ & $t^1$ & $t^2$ & $t^3$ & $t^4$ & $t^5$ & $t^6$ \\
\midrule

$\Sylv[k][0]$
& $0$
& 
&
&
&
&
& \\
\addlinespace[2pt]

$\Sylv[k][1]$
& $0$
& $k$
&
&
&
&
& \\
\addlinespace[2pt]

$\Sylv[k][2]$
& $0$
& $k$
& $k$
&
&
&
& \\
\addlinespace[3pt]

$\Sylv[k][3]$
& $0$
& $k$
& $10k + 8$
& $k$
&
&
& \\
\addlinespace[5pt]

$\Sylv[k][4]$
& $0$
& $k$
& $251k+240$
& $251k+240$
& $k$
&
& \\
\addlinespace[8pt]

$\Sylv[k][5]$
& \makecell{{$0$}}
& \makecell{$k$}
& \makecell{$151700k+151448 $}
& \makecell{$606822k + 606320$}
& \makecell{$151700k+151448 $}
& \makecell{$k$}
& \\
\addlinespace[8pt]

$\Sylv[k][6]$
& \makecell{\textcolor{white}{$0$}\\[4pt]$0$}
& \makecell{$k$\\[4pt]\textcolor{white}{$0$}}
& \makecell{$123769377141k$\\[4pt]$+123769225440$}
& \makecell{$1361462238362k$\\[4pt]$+1361461479840$}
& \makecell{$1361462238362k$\\[4pt]$+1361461479840$}
& \makecell{$123769377141k$\\[4pt]$+123769225440$}
&\makecell{$k$\\[4pt]\textcolor{white}{$0$}}
\\
\bottomrule
\end{tabular}
\end{adjustbox}
\caption{Local $h^\ast$-polynomial of $\Sylv$, \ie $\ell^\ast(\Sylv; t)$, for $0 \leq d \leq 6$.
}
\label{tab:local_h_coefficients}
\end{table}
\begin{table}[htbp]
\setlength{\tabcolsep}{3.5pt}

\begin{adjustbox}{center,max width=0.96\textheight,max totalheight=0.70\textwidth}
\begin{tabular}{@{}c c c c c c c c@{}}
\toprule
 & $t^0$ & $t^1$ & $t^2$ & $t^3$ & $t^4$ & $t^5$ & $t^6$ \\
\midrule

$\Sylv[k][0]$
& $1$
&
&
&
&
&
& \\
\addlinespace[2pt]

$\Sylv[k][1]$
& $1$
& $1$
&
&
&
&
& \\
\addlinespace[2pt]

$\Sylv[k][2]$
& $1$
& $2k+4$
& $1$
&
&
&
& \\
\addlinespace[3pt]

$\Sylv[k][3]$
& $1$
& $15k+20$
& $15k+20$
& $1$
&
&
& \\
\addlinespace[5pt]

$\Sylv[k][4]$
& $1$
& $327k+348$
& $1068k+1108$
& $327k+348$
& $1$
&
& \\
\addlinespace[8pt]

$\Sylv[k][5]$
& \makecell{$1$}
& \makecell{$177715k+178064$}
& \makecell{$1452200k+1453656$}
& \makecell{$1452200k+1453656$}
& \makecell{$177715k+178064$}
& \makecell{$1$}
& \\
\addlinespace[8pt]

$\Sylv[k][6]$
& \makecell{\textcolor{white}{$0$}\\[4pt]$1$}
& \makecell{$134562679887k$\\[4pt]$+134562857952$}
& \makecell{$2398473604116k$\\[4pt]$+2398475235836$}
& \makecell{$5583977855916k$\\[4pt]$+5583980763228$}
& \makecell{$2398473604116k$\\[4pt]$+2398475235836$}
& \makecell{$134562679887k$\\[4pt]$+134562857952$}
& \makecell{\textcolor{white}{$0$}\\[4pt]$1$}
\\
\bottomrule
\end{tabular}
\end{adjustbox}

\caption{Boundary $h^\ast$-polynomial of $\Sylv$, \ie $h^\ast(\partial \Sylv; t)$, for $0 \leq d \leq 6$.
}
\label{tab:boundary_h_coefficients}
\end{table}

\begin{table}[p]
\centering
\rotatebox{90}{%
\begin{minipage}{\textheight}
\centering
\begin{tabular}{@{}c c c c c c c c c@{}}
\toprule
 & $f^\ast_{-1}$ & $f^\ast_0$ & $f^\ast_1$ & $f^\ast_2$ & $f^\ast_3$ & $f^\ast_4$ & $f^\ast_5$ & $f^\ast_6$ \\
\midrule

$\Sylv[k][0]$
& $\textcolor{white}{0k +} \textcolor{blue}{\textbf{1}}$
& $\textcolor{white}{0k +} 1$
& 
& 
& 
&
&
& \\
\addlinespace[2pt]

$\Sylv[k][1]$
& $\textcolor{white}{0k +} 1$
& $1k + \textcolor{blue}{\textbf{2}}$
& $1k + 1$
&
&
&
&
&  \\
\addlinespace[2pt]

$\Sylv[k][2]$
& $\textcolor{white}{0k +} 1$
& $3k + 6$
& $7k + \textcolor{blue}{\textbf{9}}$
& $4k + 4$
&
&
&
&  \\
\addlinespace[3pt]

$\Sylv[k][3]$
& $\textcolor{white}{0k +} 1$
& $16k + 23$
& $67k + 79$
& $87k + \textcolor{blue}{\textbf{93}}$
& $36k + 36$
&
&
& \\
\addlinespace[5pt]

$\Sylv[k][4]$
& $\textcolor{white}{0k +} 1$
& $328k + 352$
& $2400k + 2486$
& $\textcolor{blue}{\textbf{5579}}k + \textcolor{red}{\textbf{5684}}$
& $5271k + 5313$
& $1764k + 1764$
&
& \\
\addlinespace[8pt]

$\Sylv[k][5]$
& \makecell{\\[4pt]$1$}
& \makecell{$177716k$\\[4pt]$+178069$}
& \makecell{$2341128k$\\[4pt]$+2343638$}
& \makecell{$9040660k$\\[4pt]$+9046430$}
& \makecell{$15030435k$\\[4pt]$+15035853$}
& \makecell{$11414823k$\\[4pt]$+11416629$}
& \makecell{$3261636k$\\[4pt]$+3261636$}
& \\
\addlinespace[8pt]

$\Sylv[k][6]$
& \makecell{\\[4pt]$1$}
& \makecell{$134562679888k$\\[4pt]$+134562857958$}
& \makecell{$3205849861508k$\\[4pt]$+3205852205499$}
& \makecell{$19594788418792k$\\[4pt]$+19594797467732$}
& \makecell{$51410383715620k$\\[4pt]$+51410398757243$}
& \makecell{$66837041908819k$\\[4pt]$+66837053330866$}
& \makecell{$42600213117735k$\\[4pt]$+42600216381177$}
& \makecell{$10650053687364k$\\[4pt]$+10650053687364$} \\
\bottomrule
\end{tabular}
\caption{$f^\ast$-vectors of $\Sylv[k][d]$ for $0 \leq d \leq 6$.
As an illustration of \Cref{thm:f_star_induction}, the sum of the \textcolor{blue}{\textbf{blue}} numbers is equal to the \textcolor{red}{\textbf{red}} number.}
\label{tab:f_star_coefficients}

\end{minipage}
}

\end{table}

\clearpage
\subsection{Results for \texorpdfstring{$d = 7$}{d = 7}}\label{ssec:TablesD7}
\textcolor{white}{Nothing}

\textbf{Ehrhart polynomial of $\Sylv$ for $d = 7$:}

$
\begin{array}{rrcll}
\vspace{0.1cm}
&\Bigl(\tfrac{450094099426234898031943}{20}\,k &+& \tfrac{450094099426234898031943}{20}\Bigr) &t^7  \\
\vspace{0.1cm}
+&\Bigl(\tfrac{3780790435180018141569961}{48}\,k &+& \tfrac{18903952175903640726833407}{240}\Bigr) &t^6\\
\vspace{0.1cm}
+&\Bigl(\tfrac{24744702371155698986025517}{240}\,k &+& \tfrac{24744702371166349042976323}{240}\Bigr) &t^5 \\ 
\vspace{0.1cm}
+&\Bigl(\tfrac{2920375097628417431403805}{48}\,k &+& \tfrac{973458365876913621918503}{16}\Bigr) &t^4 \\
\vspace{0.1cm}
+&\Bigl(\tfrac{1160115412419970824109333}{80}\,k &+& \tfrac{3480346237265396720927029}{240}\Bigr) &t^3\\
\vspace{0.1cm}
+&\Bigl(\tfrac{7061995569652883460737}{24}\,k &+& \tfrac{4413747231094324176263}{15}\Bigr) &t^2 \\
\vspace{0.1cm}
-&\Bigl(\tfrac{6216272723965441907549}{30}\,k &+& \tfrac{12432545447927574463817}{60}\Bigr) &t\\
\vspace{0.1cm}
&&+& 1 &

\end{array}
$

\vspace{0.25cm}
\textbf{Ehrhart polynomial of $\Sylv$ for $d = 7$ in the magic basis:}

$
\begin{array}{crcll}
\vspace{0.1cm}
&k &&&  \,t^7(1+t)^{d-7}\\
\vspace{0.1cm}
- &\Bigl(\tfrac{12432545447927574464237}{60}k &+& \tfrac{6216272723965441907549}{30}  \Bigr)  &t^6(1+t)^{d-6} \\
\vspace{0.1cm}
+&\Bigl(\tfrac{46125130805971371744607}{30}k &+& \tfrac{184500523223435023084861}{120} \Bigr) &t^5(1+t)^{d-5} \\  
\vspace{0.1cm}
+&\Bigl(\tfrac{793764577300732106329523}{80}k &+& \tfrac{2381293731901415270385269}{240} \Bigr) &t^4(1+t)^{d-4} \\
\vspace{0.1cm}
+&\Bigl(\tfrac{2381293731901415270385269}{240}k &+& \tfrac{793764577300732106329523}{80} \Bigr) &t^3(1+t)^{d-3}\\
\vspace{0.1cm}
+&\Bigl(\tfrac{184500523223435023084861}{120}k &+& \tfrac{46125130805971371744607}{30} \Bigr) &t^2(1+t)^{d-2}\\
\vspace{0.1cm}
-&\Bigl(\tfrac{6216272723965441907549}{30}k &+& \tfrac{12432545447927574464237}{60} \Bigr)&t(1+t)^{d-1}\\
\vspace{0.1cm}
+&&& 1&t^0(1+t)^{d}
\end{array}
$

\vspace{0.25cm}
\textbf{$h^\ast$-polynomial of $\Sylv$ for $d = 7$:}

$
\begin{array}{crcll}
&k &&&t^7  \\
+&(279803730164623785353587 \,k &+& 279803730164489222495636)\,&t^6 \\
+&  (10072752446273184877375736 \,k &+& 10072752446270920964997852)\, &t^5 \\  
+&(46359300351270580479677108 \,k &+& 46359300351267394974149716)\,&t^4 \\
+&(46359300351267394974149716 \,k &+& 46359300351270580479677108)\,&t^3 \\
+&(10072752446270920964997852 \,k &+& 10072752446273184877375736)\,&t^2 \\
+&(279803730164489222495636 \,k &+& 279803730164623785353587)\,&t\\
&&& 1&

\end{array}
$

\vspace{0.25cm}
\textbf{$f^\ast$-vector of $\Sylv$ for $d = 7$:}

$
\begin{array}{lrcll}
f_{-1}^\ast(\Sylv) = &&&1 &  \\
f_{0}^\ast(\Sylv[k][7]) = &279803730164489222495636 \,k &+&  279803730164623785353595\,& \\
f_{1}^\ast(\Sylv[k][7]) = &12031378557422345522467304 \,k &+& 12031378557425551374850873\,& \\
f_{2}^\ast(\Sylv[k][7]) = &  112671693362347194436545184 \,k &+& 112671693362366789236356907\, & \\  
f_{3}^\ast(\Sylv[k][7]) = &439040219357428492612740728 \,k &+& 439040219357479903020546911\,& \\
f_{4}^\ast(\Sylv[k][7]) = &870351136845204998624885628 \,k &+& 870351136845271835693258117\,& \\
f_{5}^\ast(\Sylv[k][7]) = &929214031716799699184416739 \,k &+& 929214031716842299412219963\,& \\
f_{6}^\ast(\Sylv[k][7]) = &510406708749345049339747959 \,k &+& 510406708749355699396698765\,&\\
f_{7}^\ast(\Sylv[k][7]) = &113423713055411194304049636\, k& + & 113423713055411194304049636&

\end{array}
$

\addtocontents{toc}{\protect\setcounter{tocdepth}{2}}

\bibliographystyle{alpha}
\bibliography{biblio.bib}

@article{balletti2021enumeration,
	title={Enumeration of lattice polytopes by their volume},
	author={Balletti, Gabriele},
	journal={Discrete \& Computational Geometry},
	volume={65},
	number={4},
	pages={1087--1122},
	year={2021},
	publisher={Springer}
}

@article{Kasprzyk,
 author  = {Kasprzyk, Alexander M.},
 title   = {Canonical Toric {F}ano Threefolds},
 journal = {Canadian Journal of Mathematics},
 volume  = {62},
 number  = {6},
 pages   = {1293--1309},
 year    = {2010},
 doi     = {10.4153/CJM-2010-070-3}
}

@book {BeckSanyal,
    AUTHOR = {Beck, Matthias and Sanyal, Raman},
     TITLE = {Combinatorial reciprocity theorems},
    SERIES = {Graduate Studies in Mathematics},
    VOLUME = {195},
      NOTE = {An invitation to enumerative geometric combinatorics},
 PUBLISHER = {American Mathematical Society, Providence, RI},
      YEAR = {2018},
     PAGES = {xiv+308},
      ISBN = {978-1-4704-2200-4},
   MRCLASS = {05-01 (05A15 05C30 06A07 11P21 52B05 52C35)},
  MRNUMBER = {3839322},
MRREVIEWER = {Philippe\ Nadeau},
       DOI = {10.1090/gsm/195},
       URL = {https://doi.org/10.1090/gsm/195},
}

@article{lagarias1991bounds,
	title={Bounds for lattice polytopes containing a fixed number of interior points in a sublattice},
	author={Lagarias, Jeffrey C. and Ziegler, G{\"u}nter M},
	journal={Canadian Journal of Mathematics},
	volume={43},
	number={5},
	pages={1022--1035},
	year={1991},
	publisher={Cambridge University Press}
}

@article{hensley1983lattice,
	title={Lattice vertex polytopes with interior lattice points},
	author={Hensley, Douglas},
	journal={Pacific Journal of Mathematics},
	volume={105},
	number={1},
	pages={183--191},
	year={1983},
	publisher={Mathematical Sciences Publishers}
}

@article{wills1982lattice,
	title={On lattice polytopes having interior lattice points.},
	author={Wills, J{\"o}rg M. and Zaks, Joseph and Perles, Micha A.},
	journal={Elemente der Mathematik},
	volume={37},
	pages={44--45},
	year={1982}
}

@article{averkov2020local,
	title={Local optimality of {Z}aks--{P}erles--{W}ills simplices},
	author={Averkov, Gennadiy},
	journal={Advances in Applied Mathematics},
	volume={112},
	pages={101943},
	year={2020},
	publisher={Elsevier}
}

@article{averkov2015largest,
	title={Largest integral simplices with one interior integral point: {S}olution of {H}ensley's conjecture and related results},
	author={Averkov, Gennadiy and Kr{\"u}mpelmann, Jan and Nill, Benjamin},
	journal={Advances in Mathematics},
	volume={274},
	pages={118--166},
	year={2015},
	publisher={Elsevier}
}

@article{averkov2020lattice,
	title={Lattice simplices with a fixed positive number of interior lattice points: A nearly optimal volume bound},
	author={Averkov, Gennadiy and Kr{\"u}mpelmann, Jan and Nill, Benjamin},
	journal={International Mathematics Research Notices},
	volume={2020},
	number={13},
	pages={3871--3885},
	year={2020},
	publisher={Oxford University Press}
}

@misc{balletti2016three,
      title={Three-dimensional lattice polytopes with two interior lattice points}, 
      author={Gabriele Balletti and Alexander M. Kasprzyk},
      year={2016},
      eprint={1612.08918},
      archivePrefix={arXiv},
      primaryClass={math.CO},
      url={https://arxiv.org/abs/1612.08918}, 
}

@misc{CaicedoJuhnkePoullot,
title={Ehrhart non-positivity and unimodular triangulations for classes of $s$-lecture hall simplices},
author={Jhon B. Caicedo and Martina Juhnke and Germain Poullot},
year={2025},
eprint={2508.18890},
archivePrefix={arXiv},
primaryClass={math.CO},
url={https://arxiv.org/abs/2508.18890},
}

@article{Balletti_dual_volume, 
title={On the maximum dual volume of a canonical {F}ano polytope}, 
volume={10}, 
DOI={10.1017/fms.2022.93}, 
journal={Forum of Mathematics, Sigma}, 
author={Balletti, Gabriele and Kasprzyk, Alexander M. and Nill, Benjamin}, 
year={2022}, 
pages={e109}
}

@article {Unimodular_triangulation_existence,
    AUTHOR = {Haase, Christian and Paffenholz, Andreas and Piechnik, Lindsay
              C. and Santos, Francisco},
     TITLE = {Existence of unimodular triangulations---positive results},
   JOURNAL = {Mem. Amer. Math. Soc.},
  FJOURNAL = {Memoirs of the American Mathematical Society},
    VOLUME = {270},
      YEAR = {2021},
    NUMBER = {1321},
     PAGES = {v+83},
      ISSN = {0065-9266,1947-6221},
      ISBN = {978-1-4704-4716-8; 978-1-4704-6530-8},
   MRCLASS = {52B20 (13F20 13P10 14M25)},
  MRNUMBER = {4277268},
MRREVIEWER = {Margaret\ M.\ Bayer},
       DOI = {10.1090/memo/1321},
       URL = {https://doi.org/10.1090/memo/1321},
}

@book {Beck_book,
    AUTHOR = {Beck, Matthias and Robins, Sinai},
     TITLE = {Computing the continuous discretely},
    SERIES = {Undergraduate Texts in Mathematics},
   EDITION = {Second},
      NOTE = {Integer-point enumeration in polyhedra,
              With illustrations by David Austin},
 PUBLISHER = {Springer, New York},
      YEAR = {2015},
     PAGES = {xx+285},
      ISBN = {978-1-4939-2968-9; 978-1-4939-2969-6},
   MRCLASS = {11P21 (05A15 05B15 11-02 11H06 52B05 52B20)},
  MRNUMBER = {3410115},
       DOI = {10.1007/978-1-4939-2969-6},
       URL = {https://doi.org/10.1007/978-1-4939-2969-6},
}

@unpublished {OEIS,
  	KEY = {{OEIS}},
	TITLE = {The {O}n-{L}ine {E}ncyclopedia of {I}nteger {S}equences},
	NOTE = {Published electronically at \url{http://oeis.org}},
	YEAR = {2010},
}

@misc{FerroniHigashitani2024ExamplesCounterexamplesEhrhartTheory,
      title={Examples and counterexamples in {E}hrhart theory}, 
      author={Luis Ferroni and Akihiro Higashitani},
      year={2024},
      eprint={2307.10852},
      archivePrefix={arXiv},
      primaryClass={math.CO},
      url={https://arxiv.org/abs/2307.10852}, 
}

@article {Nill2007VolumeLatticePointsReflexiveSimplices,
    AUTHOR = {Nill, Benjamin},
     TITLE = {Volume and lattice points of reflexive simplices},
   JOURNAL = {Discrete Comput. Geom.},
  FJOURNAL = {Discrete \& Computational Geometry. An International Journal
              of Mathematics and Computer Science},
    VOLUME = {37},
      YEAR = {2007},
    NUMBER = {2},
     PAGES = {301--320},
      ISSN = {0179-5376,1432-0444},
   MRCLASS = {14M25 (14J45 52B20)},
  MRNUMBER = {2295061},
MRREVIEWER = {Alexander\ M.\ Kasprzyk},
       DOI = {10.1007/s00454-006-1299-y},
       URL = {https://doi.org/10.1007/s00454-006-1299-y},
}

@article {Magic_Petter,
    AUTHOR = {Br{\"a}nd{\'e}n, Petter},
     TITLE = {On linear transformations preserving the {P}\'olya frequency
              property},
   JOURNAL = {Trans. Amer. Math. Soc.},
  FJOURNAL = {Transactions of the American Mathematical Society},
    VOLUME = {358},
      YEAR = {2006},
    NUMBER = {8},
     PAGES = {3697--3716},
      ISSN = {0002-9947,1088-6850},
   MRCLASS = {05A15 (05A05 05A19 20F55 26C10)},
  MRNUMBER = {2218995},
MRREVIEWER = {Toshihiro\ Watanabe},
       DOI = {10.1090/S0002-9947-06-03856-6},
       URL = {https://doi.org/10.1090/S0002-9947-06-03856-6},
}

@article{Athanasiadis2004-hStarVectorsAndEulerianPolynomials,
    AUTHOR = {Athanasiadis, Christos A.},
     TITLE = {{$h^\ast$}-vectors, {E}ulerian polynomials and stable
              polytopes of graphs},
   JOURNAL = {Electron. J. Combin.},
  FJOURNAL = {Electronic Journal of Combinatorics},
    VOLUME = {11},
      YEAR = {2004/06},
    NUMBER = {2},
     PAGES = {Research Paper 6, 13},
      ISSN = {1077-8926},
   MRCLASS = {05E25 (05C17 06A11 13A99 52B20)},
  MRNUMBER = {2120101},
MRREVIEWER = {Jaroslav\ Ivan\v co},
       DOI = {10.37236/1863},
       URL = {https://doi.org/10.37236/1863},
}

@article{Reeve1957,
    AUTHOR = {Reeve, John E.},
     TITLE = {On the volume of lattice polyhedra},
   JOURNAL = {Proc. London Math. Soc. (3)},
  FJOURNAL = {Proceedings of the London Mathematical Society. Third Series},
    VOLUME = {7},
      YEAR = {1957},
     PAGES = {378--395},
      ISSN = {0024-6115,1460-244X},
   MRCLASS = {52.00 (10.00)},
  MRNUMBER = {95452},
MRREVIEWER = {H.\ S. M. Coxeter},
       DOI = {10.1112/plms/s3-7.1.378},
       URL = {https://doi.org/10.1112/plms/s3-7.1.378},
}

@article{BeckDeligeorgakiHlavacekValenciaPorras2024-fStarVectors,
    AUTHOR = {Beck, Matthias and Deligeorgaki, Danai and Hlavacek, Max and
              Valencia-Porras, Jer\'onimo},
     TITLE = {Inequalities for {$f^*$}-vectors of lattice polytopes},
   JOURNAL = {Adv. Geom.},
  FJOURNAL = {Advances in Geometry},
    VOLUME = {24},
      YEAR = {2024},
    NUMBER = {2},
     PAGES = {141--150},
      ISSN = {1615-715X,1615-7168},
   MRCLASS = {52B20 (05A15 05A20)},
  MRNUMBER = {4737131},
MRREVIEWER = {Ruriko\ Yoshida},
       DOI = {10.1515/advgeom-2024-0002},
       URL = {https://doi.org/10.1515/advgeom-2024-0002},
}

@article{BetkeMcMullen1985-LatticePoints,
    AUTHOR = {Betke, Ulrich and McMullen, Peter},
     TITLE = {Lattice points in lattice polytopes},
   JOURNAL = {Monatsh. Math.},
  FJOURNAL = {Monatshefte f\"ur Mathematik},
    VOLUME = {99},
      YEAR = {1985},
    NUMBER = {4},
     PAGES = {253--265},
      ISSN = {0026-9255,1436-5081},
   MRCLASS = {52A43 (11H06 52A25)},
  MRNUMBER = {799674},
MRREVIEWER = {J.\ M.\ Wills},
       DOI = {10.1007/BF01312545},
       URL = {https://doi.org/10.1007/BF01312545},
}

@article{BajoBeck2023-BoundaryhStar,
    AUTHOR = {Bajo, Esme and Beck, Matthias},
     TITLE = {Boundary {$h^\ast$}-polynomials of rational polytopes},
   JOURNAL = {SIAM J. Discrete Math.},
  FJOURNAL = {SIAM Journal on Discrete Mathematics},
    VOLUME = {37},
      YEAR = {2023},
    NUMBER = {3},
     PAGES = {1952--1969},
      ISSN = {0895-4801,1095-7146},
   MRCLASS = {52B20 (05A15 52C07)},
  MRNUMBER = {4633746},
MRREVIEWER = {Michael\ von Thaden},
       DOI = {10.1137/22M1508911},
       URL = {https://doi.org/10.1137/22M1508911},
}

@article {Ehrhart_poly,
    AUTHOR = {Ehrhart, Eug\`ene},
     TITLE = {Sur les poly\`edres rationnels homoth\'etiques \`a{} {$n$}\
              dimensions},
   JOURNAL = {C. R. Acad. Sci. Paris},
  FJOURNAL = {Comptes Rendus Hebdomadaires des S\'eances de l'Acad\'emie des
              Sciences},
    VOLUME = {254},
      YEAR = {1962},
     PAGES = {616--618},
      ISSN = {0001-4036},
   MRCLASS = {10.25 (52.10)},
  MRNUMBER = {130860},
}

@article{stanley1980decompositions,
AUTHOR = {Stanley, Richard P.},
     TITLE = {Decompositions of rational convex polytopes},
   JOURNAL = {Ann. Discrete Math.},
  FJOURNAL = {Annals of Discrete Mathematics},
    VOLUME = {6},
      YEAR = {1980},
     PAGES = {333--342},
   MRCLASS = {52A43},
  MRNUMBER = {593545},
MRREVIEWER = {P.\ McMullen},
}

@article{Stanley1974,
  author  = {Stanley, Richard P.},
  title   = {Combinatorial Reciprocity Theorems},
  journal = {Advances in Mathematics},
  volume  = {14},
  number  = {2},
  pages   = {194--253},
  year    = {1974},
  doi     = {10.1016/0001-8708(74)90030-9}
}

@article {Local_h_definition,
    AUTHOR = {Katz, Eric and Stapledon, Alan},
     TITLE = {Local {$h$}-polynomials, invariants of subdivisions, and mixed
              {E}hrhart theory},
   JOURNAL = {Adv. Math.},
  FJOURNAL = {Advances in Mathematics},
    VOLUME = {286},
      YEAR = {2016},
     PAGES = {181--239},
      ISSN = {0001-8708,1090-2082},
   MRCLASS = {52B20 (05A15 06A07)},
  MRNUMBER = {3415684},
MRREVIEWER = {Matthias\ Beck},
       DOI = {10.1016/j.aim.2015.09.010},
       URL = {https://doi.org/10.1016/j.aim.2015.09.010},
}

@incollection {box_polynomial_definition,
    AUTHOR = {Braun, Benjamin},
     TITLE = {Unimodality problems in {E}hrhart theory},
 BOOKTITLE = {Recent trends in combinatorics},
    SERIES = {IMA Vol. Math. Appl.},
    VOLUME = {159},
     PAGES = {687--711},
 PUBLISHER = {Springer, [Cham]},
      YEAR = {2016},
      ISBN = {978-3-319-24296-5; 978-3-319-24298-9},
   MRCLASS = {52B20 (05A20)},
  MRNUMBER = {3526428},
MRREVIEWER = {Ruriko\ Yoshida},
       DOI = {10.1007/978-3-319-24298-9\_27},
       URL = {https://doi.org/10.1007/978-3-319-24298-9_27},
}

@misc{boundaryhunimodulartriangulations,
      title={Boundary {$h^\ast$}-vectors and unimodular triangulations}, 
      author={Martina Juhnke and Steffen Schlie},
      year={2026},
      eprint={2604.24377},
      archivePrefix={arXiv},
      primaryClass={math.CO},
      url={https://arxiv.org/abs/2604.24377}, 
}

@article {Felix_f_star,
    AUTHOR = {Breuer, Felix},
     TITLE = {Ehrhart {$f^*$}-coefficients of polytopal complexes are
              non-negative integers},
   JOURNAL = {Electron. J. Combin.},
  FJOURNAL = {Electronic Journal of Combinatorics},
    VOLUME = {19},
      YEAR = {2012},
    NUMBER = {4},
     PAGES = {Paper 16, 22},
      ISSN = {1077-8926},
   MRCLASS = {52B70 (05E45 11B75 52B20)},
  MRNUMBER = {3001653},
MRREVIEWER = {Steven\ Klee},
       DOI = {10.37236/2106},
       URL = {https://doi.org/10.37236/2106},
}

@incollection{Branden2015-UnimodalityLogConcavityRealRootednessGammaPositivity,
    AUTHOR = {Br{\"a}nd{\'e}n, Petter},
     TITLE = {Unimodality, log-concavity, real-rootedness and beyond},
 BOOKTITLE = {Handbook of enumerative combinatorics},
    SERIES = {Discrete Math. Appl. (Boca Raton)},
     PAGES = {437--483},
 PUBLISHER = {CRC Press, Boca Raton, FL},
      YEAR = {2015},
      ISBN = {978-1-4822-2085-8},
   MRCLASS = {05Axx (05D40 05E10)},
  MRNUMBER = {3409348},
}

@manual {Sage,
	AUTHOR = {{{S}age developers}},
	TITLE = {{S}age {M}athematics {S}oftware},
	NOTE = {\url{http://www.sagemath.org}},
	YEAR = {2024},
}

@article{Sylvester1880-OnVulgarFractions,
    AUTHOR = {Sylvester, James J.},
     TITLE = {On a {P}oint in the {T}heory of {V}ulgar {F}ractions},
   JOURNAL = {Amer. J. Math.},
  FJOURNAL = {American Journal of Mathematics},
    VOLUME = {3},
      YEAR = {1880},
    NUMBER = {4},
     PAGES = {332--335},
      ISSN = {0002-9327,1080-6377},
   MRCLASS = {99-04},
  MRNUMBER = {1505274},
       DOI = {10.2307/2369261},
       URL = {https://doi.org/10.2307/2369261},
}

@misc{ferroni2026unimodality,
      title={Unimodality shenanigans in Ehrhart theory}, 
      author={Luis Ferroni},
      year={2026},
      eprint={2609.10513},
      archivePrefix={arXiv},
      primaryClass={math.CO},
      url={https://arxiv.org/abs/2609.10513}, 
}

@article{scott1976convex,
  title={On convex lattice polygons},
  author={Scott, Paul R.},
  journal={Bulletin of the Australian Mathematical Society},
  volume={15},
  number={3},
  pages={395--399},
  year={1976},
  publisher={Cambridge University Press}
}

@article{haase2009lattice,
  title={Lattice polygons and the number $2 i+ 7$},
  author={Haase, Christian and Schicho, Josef},
  journal={The American Mathematical Monthly},
  volume={116},
  number={2},
  pages={151--165},
  year={2009},
  publisher={Taylor \& Francis}
}

@article{batyrev2006multiples,
    AUTHOR = {Batyrev, Victor and Nill, Benjamin},
     TITLE = {Multiples of lattice polytopes without interior lattice
              points},
   JOURNAL = {Mosc. Math. J.},
  FJOURNAL = {Moscow Mathematical Journal},
    VOLUME = {7},
      YEAR = {2007},
    NUMBER = {2},
     PAGES = {195--207, 349},
      ISSN = {1609-3321,1609-4514},
   MRCLASS = {52B20 (14M25)},
  MRNUMBER = {2337878},
MRREVIEWER = {Joseph\ P.\ Rusinko},
       DOI = {10.17323/1609-4514-2007-7-2-195-207},
       URL = {https://doi.org/10.17323/1609-4514-2007-7-2-195-207},
}
\label{sec:biblio}

\end{document}